\documentclass[a4paper,11pt]{article}
\usepackage{authblk}

\usepackage{comment}

\usepackage[a4paper,margin=2.5cm,includefoot]{geometry}

\usepackage{amsmath, amssymb, amsthm,mathtools}

\usepackage[alphabetic]{amsrefs}

\usepackage{dsfont}

\usepackage{hyperref}

\usepackage{stmaryrd}%for the command \mapsfrom
\usepackage[all]{xy}

\usepackage{enumitem}

\newtheorem{thm}{Theorem}[section]
\newtheorem{lem}[thm]{Lemma}
\newtheorem{defi}[thm]{Definition}
\newtheorem{prop}[thm]{Proposition}
\newtheorem{coro}[thm]{Corollary}
\theoremstyle{remark}
\newtheorem{rmk}[thm]{Remark}

\newcommand{\define}[1]{\textbf{#1}}%defined terms are in boldface

\DeclareMathOperator\inv{inv}
\DeclareMathOperator\id{id}
\DeclareMathOperator\Ad{Ad}
\DeclareMathOperator\ad{ad}
\DeclareMathOperator\Der{Der}

\DeclareMathOperator\Hom{Hom}

\DeclareMathOperator{\CE}{CE}

\DeclareMathOperator\proj{pr}

\DeclareMathOperator\eval{eval}%evaluation

\newcommand{\smoothF}{\mathcal{C}^{\infty}}%smooth functions
\newcommand{\gradedF}{\mathcal{C}}%graded functions
\newcommand{\vf}{\mathfrak{X}}%vector fields
\newcommand{\linvvf}{\mathfrak{X}^{\mathrm{left}}}%left-invariant vector fields
\newcommand{\multvf}{\mathfrak{X}^{\mathrm{mult}}}%multiplicative vector fields

\newcommand{\cat}[1]{{\tt #1}}%categories are in \tt

\newcommand{\conv}{\mathbin{\star}}%convolution product

\newcommand{\compTens}{\mathbin{\widehat{\otimes}}}%completed tensor product

\newcommand{\longisomto}{\stackrel{\sim}{\longrightarrow}}

\newcommand{\mathfrakg}{\mathfrak{g}}
\newcommand{\mathfrakh}{\mathfrak{h}}

\newcommand{\mathfrakU}{\mathfrak{U}}

\newcommand{\calC}{\mathcal{C}}

\newcommand{\calG}{\mathcal{G}}
\newcommand{\calL}{\mathcal{L}}
\newcommand{\calM}{\mathcal{M}}
\newcommand{\calN}{\mathcal{N}}
\newcommand{\calO}{\mathcal{O}}

\newcommand{\NN}{\mathbb{N}}
\newcommand{\ZZ}{\mathbb{Z}}

\newcommand{\RR}{\mathbb{R}}

\newcommand{\CCE}{C_{\mathrm{CE}}}%Chevalley--Eilenberg complex
\DeclareMathOperator{\grsym}{\mathrm{S}}%graded symmetric algebra
\newcommand{\bee}{\begin{enumerate}}
\newcommand{\eee}{\end{enumerate}}
\newcommand{\benn}{\begin{equation*}}
\newcommand{\eenn}{\end{equation*}}
\newcommand{\be}{\begin{equation}}
\newcommand{\ee}{\end{equation}}
\newcommand{\bean}{\begin{eqnarray}}
\newcommand{\eean}{\end{eqnarray}}
\newcommand{\bea}{\begin{eqnarray*}}
\newcommand{\eea}{\end{eqnarray*}}

\newcommand{\N}{\mathbb{N}}
\newcommand{\Z}{\mathbb{Z}}
\newcommand{\R}{\mathbb{R}}

\newcommand{\C}{\mathbb{C}}

\newcommand{\op}[1]{\!\!\mathop{\rm ~#1}\nolimits}

\mathchardef\za="710B  %\alpha
\mathchardef\zb="710C  %\beta
\mathchardef\zg="710D  %\gamma
\mathchardef\zd="710E  %\delta
\mathchardef\zve="710F %\epsilon
\mathchardef\zz="7110  %\zeta
\mathchardef\zh="7111  %\eta
\mathchardef\zy="7112 %\theta

\mathchardef\zi="7113  %\iota
\mathchardef\zk="7114  %\kappa
\mathchardef\zl="7115  %\lambda
\mathchardef\zm="7116  %\mu
\mathchardef\zn="7117  %\nu
\mathchardef\zx="7118  %\xi
\mathchardef\zp="7119  %\pi
\mathchardef\zr="711A  %\rho
\mathchardef\zs="711B  %\sigma
\mathchardef\zt="711C  %\tau
\mathchardef\zu="711D  %\upsilon
\mathchardef\zf="711E %\phi
\mathchardef\zq="711F  %\chi
\mathchardef\zc="7120  %\psi
\mathchardef\zw="7121  %\omega
\mathchardef\ze="7122  %\varepsilon
\mathchardef\zvy="7123  %\vartheta
\mathchardef\zvw="7124  %\varomega
\mathchardef\zvr="7125 %\varrho
\mathchardef\zvs="7126 %\varsigma
\mathchardef\zvf="7127  %\varphi
\mathchardef\zG="7000  %\Gamma
\mathchardef\zD="7001  %\Delta
\mathchardef\zY="7002  %\Theta
\mathchardef\zL="7003  %\Lambda
\mathchardef\zX="7004  %\Xi
\mathchardef\zP="7005  %\Pi
\mathchardef\zS="7006  %\Sigma
\mathchardef\zU="7007  %\Upsilon
\mathchardef\zF="7008  %\Phi
\mathchardef\zW="700A  %\Omega

\newcommand{\cyclic}{\mathop{\kern0.9ex{{+}
\kern-2.15ex\raise-.25ex\hbox{\Large\hbox{$\circlearrowright$}}}}\limits}

\newcommand{\bm}[1]{{\bf #1}}

\usepackage{tikz}
\usetikzlibrary{arrows,chains,matrix,positioning,scopes,snakes}
\makeatletter
\tikzset{join/.code=\tikzset{after node path={%
\ifx\tikzchainprevious\pgfutil@empty\else(\tikzchainprevious)%
edge[every join]#1(\tikzchaincurrent)\fi}}}
\makeatother

\tikzset{>=stealth',every on chain/.append style={join},
         every join/.style={->}}
\tikzset{
    >=stealth',
    punkt/.style={
           rectangle,
           rounded corners,
           draw=black, very thick,
           text width=6.5em,
           minimum height=2em,
           text centered},
    pil/.style={
           ->,
           thick,
           shorten <=2pt,
           shorten >=2pt,}
}

\title{Differential graded Lie groups\\ and their\\ differential graded Lie algebras}
\date{\today}

\author{Benoit Jubin\thanks{benoit.jubin@imj-prg.fr}}
\affil{Institut de Math\'ematiques de Jussieu - Paris Rive Gauche, \protect\\
 4 place Jussieu,
B.C. 247,
75252 Paris Cedex 5, France}

\author{Alexei Kotov\thanks{oleksii.kotov@uhk.cz}}
\affil{Faculty of Science, University of Hradec Kralove,
 Rokitanskeho 62, \protect\\
Hradec Kralove 50003, Czech Republic}

\author{Norbert Poncin\thanks{norbert.poncin@uni.lu}}
\affil{RMATH, FSTC,
Universit\'e du Luxembourg,
Maison du Nombre \protect\\
6, Avenue de la Fonte,
L-4364 Esch-sur-Alzette, Luxembourg}

\author{Vladimir Salnikov\thanks{vladimir.salnikov@univ-lr.fr}}
\affil{LaSIE  -- CNRS \&  La Rochelle University, \protect\\
Av. Michel Cr\'epeau, 17042 La Rochelle Cedex 1, France}

\begin{document}

\maketitle

\begin{abstract}
In this paper we discuss the question of integrating differential graded Lie algebras ({\small DGLA}) to differential graded Lie groups ({\small DGLG}).

We first recall the classical problem of integration in the context, and present the construction for (non-graded) differential Lie algebras.
Then, we define the category of differential graded Lie groups and study its properties.
We show how to associate a differential graded Lie algebra to every differential graded Lie group and vice-versa. For the {\small DGLA} $\to$ {\small DGLG} direction,  the main ``tools'' are graded Hopf algebras and Harish-Chandra  pairs ({\small HCP}) -- we define the category of graded and differential graded {\small HCP}s and explain how those are related to the desired construction.
We describe some near at hand examples and mention possible generalizations.

\end{abstract}

\newpage

\tableofcontents

\setlength{\parskip}{3mm}

\newpage
%%%%%%%%%%%%%%%%%%%%%%%%%%%%%%%%%%%%%%%%%%%%%%%%%%%%%%%%%%%%%
\section{Introduction}

Any finite-dimensional real Lie algebra can be integrated to a unique simply connected Lie group. This theorem of Lie and Cartan triggered a whole series of works. \\
(i) The result is false in infinite dimension (see \cite{vanKort}), but true locally in the Banach case. Recently, C.~Wockel and C.~Zhu (\cite{wockZhu}) integrated a large class of infinite-dimensional Lie algebras to {\'e}tale Lie 2-groups. \\
(ii) S.~Covez showed (\cite{covez}) that Leibniz algebras can be integrated locally to local (pointed, augmented) Lie racks. \\
(iii) M.~Crainic and R.~L.~Fernandes (\cite{craiFer}) found the obstruction to the integrability of Lie algebroids in terms of their monodromy groups, and integrated the integrable ones to unique source-simply-connected Lie groupoids (see also \cite{catFel}, \cite{severa}). H.-H.~Tseng and C.~Zhu (\cite{tsengZhu}) integrated all Lie algebroids to stacky Lie groupoids (see also \cite{wein}). \\
(iv) As for vertical categorification and homotopification, $L_\infty$-algebras were integrated by E.~Getzler (\cite{Get09}) in the nilpotent case and by A.~Henriques (\cite{henriques}) in the general case. In Getzler's approach, the integrating object is a simplicial subset of the set of Maurer-Cartan elements of the algebra. In good cases, it is a higher groupoid generalizing the Deligne groupoid of a {\small DGLA}. Recently, Y.~Sheng and C.~Zhu (\cite{shengZhu}) gave a more explicit integration for strict Lie 2-algebras (and their morphisms); their integration is Morita equivalent to Getzler's and Henriques'.

This text is the first of a series of papers, in which we intend to suggest an integration technique for infinity algebras and their morphisms, which is based on homotopy transfer and leads to concrete and explicit integrating objects. More precisely, in Getzler's work \cite{Get09}, the integrating simplicial set $\zg_\bullet(L)$ of a nilpotent Lie infinity algebra $L$ is homotopy equivalent to the Kan complex $\op{MC}_\bullet(L):=\op{MC}(L\otimes \zW_\bullet)$ whose $n$-simplices are the Maurer-Cartan elements of the homotopy Lie algebra obtained by tensoring $L$ with the {\small DGCA} $\zW_n:=\zW(\zD^n)$ of polynomial differential forms of the standard $n$-simplex, see also \cite{KPQ14}. If $L$ is concentrated in degrees $k\ge -\ell$ (resp., $-\ell\le k\le 0$), the integrating $\zg_\bullet(L)$ is a weak $\ell$-groupoid (resp., $\ell$-group). Our objective is to integrate a Lie infinity algebra by a kind of $A$-infinity group. As we have in mind homotopy transfer, the first goal is to integrate a differential graded Lie algebra ({\small DGLA}) by a differential graded Lie group ({\small DGLG}), which is the subject of the current paper.
 Surprisingly it turned out that this task hides more interesting details than expected.
Already the very definition of a {\small DGLG} is not entirely obvious. The present paper is a rigorous approach to this integration problem.

It is worth mentioning that differential graded Lie groups naturally appear in the context of characteristic classes (\cites{KS07,Kotov2018,Salerno2010} which in turn have interesting applications  in gauge theory \cites{SalStr13,KSS14, Sal15}); this motivated two of the authors to look at this subject in more details.

\newpage
{\bf Organization}. The article is organized as follows. \\
The next section (\ref{IntNGDLA}) addresses the integration problem in the case of classical (non-graded) differential Lie groups and algebras.
It is already presented in the way  suitable for generalization. In sections (\ref{GLG}) and (\ref{DGLG}) the construction is extended to the graded case, namely the graded Lie groups/algebras and differential graded Lie groups/algebras are defined.
Section (\ref{HC}) is the core of the paper where the relation between  {\small DGLG}s and {\small DGLA}s is discussed. The main ``tool'' introduced  there is graded and differential graded Harish-Chandra pairs ({\small HCP}). The key result is given by the two theorems
(\ref{thm:main}) and (\ref{thm:main2}) about equivalences of categories, establishing the
relations {\small GLG} $\leftrightarrow$ {\small GHCP} $\leftrightarrow$ {\small GLA} and {\small DGLG} $\leftrightarrow$ {\small DGHCP} $\leftrightarrow$ {\small DGLA} respectively.

{\bf Conventions}. Manifolds are second countable Hausdorff, finite-dimensional, and real. (Super) manifolds are smooth and finite-dimensional, maps between them, vector fields are smooth. (Super) Lie algebras are finite-dimensional and real, $\Z$-graded Lie algebras have finite-dimensional homogeneous components and are non-negatively (or non-positively) graded, unless the contrary is stated.

\bigskip
\newpage

\section{Integration of a (non-graded) differential Lie algebra}\label{IntNGDLA}

A {\small DGLA} is a {\small GLA} endowed with a square 0 degree 1 derivation. In the non-graded case, the concept reduces to a {\small LA} $\mathfrak{g}$ together with a derivation $\zd\in\op{Der}(\mathfrak g)$. On the global side, a {\small DGLG} is a group object in the category of differential graded manifolds. Here the word ``graded'' is (by a little abuse)
employed in contrast to ``super''.
For the $\Z$- (resp., $\N$) graded case
those are called  $\Z Q$- (resp., $\N Q$-) manifolds, i.e. $\Z$- ($\N$-)
manifolds equipped with a homological vector field $Q$, that is, a degree 1 derivation of the function algebra that Lie commutes with itself. We will give details for all of these concepts in
sections \ref{GLG} and \ref{DGLG}.
If we forget the grading, we deal with a group object in the category of manifolds equipped with a vector field. Such an object is a Lie group with a selected vector field that is compatible with the group maps, or, as we will see, a Lie group $G$ endowed with a multiplicative vector field $X\in\mathfrak{X}^{\op{mult}}(G)$. We thus have to show that differentiation and integration allow to pass from a non-graded differential Lie group $(G,X)$ ({\small DLG}) to a non-graded differential Lie algebra $(\mathfrak{g},\zd)$ ({\small DLA}) and vice versa.

\subsection{Multiplicative vector fields on Lie groups}

The goal of this subsection is to define multiplicative vector fields on a Lie semigroup (vector fields on a manifold that are compatible with the multiplication), and to show that, when defined on a Lie monoid (resp., a Lie group), they are automatically compatible with the unit (resp., with the inversion). More abstractly, Lie semigroups (resp., Lie monoids, Lie groups) endowed with a multiplicative vector field will turn out to be exactly semigroup (resp., monoid, group) objects in the category of manifolds endowed with a selected vector field.

We denote by $\tt MVF$ the category of manifolds $M,N,\ldots$ endowed with a distinguished vector field $X,Y,\ldots$ The morphisms $f:(M,X)\to (N,Y)$ of this category are the (smooth) maps $f:M\to N$ that relate $X$ and $Y$. Let us recall that $X$ is $f$-related to $Y$ -- we write $X \sim_f Y$ -- if $$Y \circ f = Tf \circ X\;.$$ Clearly $X \sim_{\id_M} X$ and, if $X \sim_f Y$ and $Y \sim_g Z$, then $X \sim_{g \circ f} Z$; hence, manifolds with a chosen vector field and maps relating them do form a category.

Semigroup, monoid, or group objects can be defined in cartesian categories $\tt C$, i.e. categories with finite products $\times$. Such a category admits a terminal object $\{\ast\}$ (indeed, since products are limits, i.e. universal cones, it is easily understood that the product of the empty family is terminal). A cartesian category $({\tt C},\times,\{\ast\})$ is thus canonically monoidal. A monoidal category of this type is called {\it cartesian monoidal}. The category $\tt MVF$ is cartesian monoidal, with product $$(M,X) \times (N,Y) = (M \times N, (X,Y))$$ and terminal object $(\{\ast\},0)$.

We are now prepared to define multiplicative vector fields on a Lie semigroup $G$. The space of vector fields (resp., multiplicative vector fields) on a manifold $G$ (resp., semigroup $G$) will be denoted by $\mathfrak{X}(G)$ (resp., $\multvf(G)$).

\begin{defi}%[Multiplicative vector field]
\label{defi:multvf}
A \define{multiplicative vector field} on a Lie semigroup $G$ is a vector field $X\in\mathfrak{X}(G)$ that is compatible with the multiplication $m:G\times G\to G$:
\begin{equation}
\label{eq:multvf1}
\multvf(G) = \left\{
X \in \vf(G):
(X,X) \sim_m X
\right\}\;.
\end{equation}
\end{defi}

Let $L_g$ and $R_g$, $g\in G$, be as usual the left and right multiplications by $g$. As well-known, the tangent map of $m$ is given by $$T_{(g,h)}m (v,w) = T_hL_g(w) + T_gR_h(v)\;,$$ if $g,h\in G$, $v \in T_gG$ and $w \in T_hG$. Therefore, $(X,X) \sim_m X$ reads
\begin{equation}
\label{eq:multvf}
X_{gh} = T_hL_g(X_h) + T_gR_h(X_g).
\end{equation}

If $G$ is a Lie monoid, its unit can be seen as a morphism $e \colon \{\ast\} \to G$. Hence, the compatibility condition of a vector field $X\in\mathfrak{X}(G)$ with the unit reads $0 \sim_e X$, i.e. $X_e = 0$.
If $G$ is a Lie group and $\inv \colon G \to G$ its inversion, the compatibility condition of $X$ with $\inv$ is $X \sim_{\inv} X$. The tangent map of $\inv$ is given by
$$T_g\inv = - T_eL_{g^{-1}} \circ T_g R_{g^{-1}},\;\text{or, symbolically, by}\;\, T_g\inv(v) = - g^{-1}\cdot v\cdot g^{-1},\; \text{if}\,\; v \in T_gG$$ (it suffices to differentiate the identity $m\circ(\id\times \inv)\circ\zD=e$, where $\zD$ is the diagonal map $\zD:G\ni g\mapsto (g,g)\in G\times G$ and $e$ the constant map $e:G\ni g\mapsto e\in G$). Hence, the compatibility condition $X\sim_{\inv}X$ reads $$X_{g^{-1}}=- T_eL_{g^{-1}}(T_g R_{g^{-1}}(X_g))\;,$$ for all $g\in G$.

\begin{prop}
\label{prop:multComp}
A multiplicative vector field on a Lie monoid (resp., Lie group) is compatible with the unit (resp., the inversion).
\end{prop}

\begin{proof}
Setting $h = e$ in Equation~\eqref{eq:multvf} gives $X_g = T_eL_g(X_e) + X_g$, and since $T_eL_g$ is an isomorphism, this implies $X_e = 0$.
Setting $h = g^{-1}$ in Equation~\eqref{eq:multvf} gives $0 = X_e = T_{g^{-1}}L_g(X_{g^{-1}}) + T_gR_{g^{-1}}(X_g)$, that is $X \sim_{\inv} X$.
\end{proof}

%Rk: the notion of differential object in a category with translation (or suspension) ``works'' only on the algebra side, not on the manifold side

%http://ncatlab.org/nlab/show/cartesian+monoidal+category
%in a cartesian monoidal caregory, each object has a diagonal embedding and an augmentation, so any object is a comonoid
%then cartesian functor = strong monoidal functor

%http://ncatlab.org/nlab/show/semicartesian+monoidal+category
%e.g. Poisson manifolds

\begin{coro}
The category of Lie semigroups (resp., monoids, groups) endowed with a multiplicative vector field (and morphisms relating them) is isomorphic to the category of semigroup (resp., monoid, group) objects in the category $\tt MVF$.
\end{coro}

For instance, a monoid object in $\tt MVF$ is an object $(G,X)\in\tt MVF$, i.e. a manifold $G$ and a vector field $X\in\mathfrak{X}(G)$, endowed with a monoidal structure, i.e. a morphism $m:(G,X)\times (G,X)\to (G,X)$ and a morphism $e:(\{\ast\},0)\to (G,X)$ (that verify the usual associativity and unitality conditions). In view of what has been said above, this is exactly a Lie monoid endowed with a multiplicative vector field.

\subsection{Van Est isomorphism}

In the following, we assume that $G$ is a Lie group with Lie algebra $\mathfrak{g}$ and unit $e$. In view of the isomorphisms $$ T_gR_{g^{-1}}:T_gG\to \mathfrak{g}\;,$$ $g\in G$, a vector field $$X:G\ni g\mapsto X_g\in T_gG\subset TG$$ can be interpreted as a smooth map $$\xi:G\ni g\mapsto \xi(g):=X_g\cdot g^{-1}:=T_gR_{g^{-1}}(X_g)\in\mathfrak{g}\;.$$ It turns out that $X$ is multiplicative if and only if $\xi$ is a 1-cocycle of $G$ valued in the adjoint representation $\op{Ad}:G\to \op{Aut}(\mathfrak g$). The tangent map $T_e\xi\in \op{End}(\mathfrak{g})$ of this Lie group 1-cocycle is a Lie algebra 1-cocycle of $\mathfrak g$ valued in the adjoint representation $\op{ad}:\mathfrak{g}\to \op{Der}(\mathfrak{g})$. Conversely, any Lie algebra 1-cocycle is obtained as the tangent map of a unique Lie group 1-cocycle. This result is known as the van Est isomorphism. However, to our knowledge, no simple proof can be found in the literature.

In the present subsection, we detail the results summarized in the preceding paragraph. For the sake of completeness, we recall the definitions of (smooth) Lie group cohomology and Lie algebra cohomology in Appendix \ref{AppCoho}.

\subsubsection*{Isomorphism between multiplicative vector field and group 1-cocycles}

For any $g\in G$, set $$\zw^R_g:=T_gR_{g^{-1}}=(T_eR_g)^{-1}\in\op{Iso}(T_gG,\mathfrak{g})\subset T_g^*G\otimes\mathfrak{g}\;.$$ The map $\zw^R$ is a $\mathfrak g$-valued $1$-form on $G$, known as the right Maurer-Cartan form of $G$ (the right-invariant $\mathfrak g$-valued 1-form on $G$ equal to identity at $e$). It can of course be viewed as a $\smoothF(G)$-linear map $$\zw^R:\mathfrak{X}(G)\longisomto \smoothF(G,\mathfrak g)=: C^1_{\op{sm}}(G,\Ad)$$ valued in the $1$-cochains of the smooth cohomology of $G$ endowed with its adjoint representation; it is clear that $\zw^R$ is a $\smoothF(G)$-module and an $\R$-vector space isomorphism with inverse implemented by the $(\zw^R_g)^{-1}$, $g\in G$.

\begin{prop} \label{prop:MCiso}
The isomorphism $\zw^R$ restricts to an $\R$-vector space isomorphism
\begin{equation}
\omega^R \colon \multvf(G) \longisomto Z_{\op{sm}}^1(G,\Ad)
\end{equation} between the space of multiplicative vector fields of $G$ and the space of smooth $1$-cocycles of $(G,\Ad)$.
\end{prop}

\begin{proof} It suffices to show that if $X$ satisfies the multiplicativity condition (\ref{eq:multvf}), i.e. if $$X_{gh} = T_hL_g(X_h) + T_gR_h(X_g)\;,$$ then $\xi:=\zw^R(X)$ satisfies the 1-cocycle condition (\ref{eq:cocycle}), i.e. $$\xi(gh)=\Ad_g(\xi(h))+\xi(g)$$ (and vice versa). Let us prove this implication: $$\xi(gh)=\zw_{gh}^R(X_{gh})=(T_eR_g)^{-1}(T_gR_h)^{-1}(T_hL_g(X_h)+T_gR_h(X_g))$$ $$=(T_eR_g)^{-1}(T_gR_h)^{-1}(T_hL_g)(X_h)+\xi(g)\;.$$ The first term of the {\small RHS} reads $$(T_gR_{g^{-1}})(T_{gh}R_{h^{-1}})(T_hL_g)(X_h)=(T_gR_{g^{-1}})(T_eL_g)(T_hR_{h^{-1}})(X_h)=\Ad_g(\xi(h))\;.$$ Hence the result.
\end{proof}

\subsubsection*{Isomorphism between group 1-cocycles and algebra 1-cocycles}

It is clear that $$T_e:C^1_{\op{sm}}(G,\Ad):=\smoothF(G,\mathfrak{g})\to \op{End}(\mathfrak{g})=:C^1_{\op{CE}}(\mathfrak{g},\ad)$$ is an $\R$-linear map from (smooth) 1-cochains of $G$ to Chevalley-Eilenberg $1$-cochains of $\mathfrak g$.

\begin{thm}[van Est isomorphism] If the Lie group $G$ is simply connected, the tangent map $T_e$ restricts to an $\R$-vector space isomorphism
\begin{equation}
T_e \colon Z^1_{\op{sm}}(G,\Ad) \longisomto \Der(\mathfrakg)=Z^1_{\op{CE}}(\mathfrak{g},\ad)
\end{equation} between group 1-cocycles of $G$ and algebra 1-cocycles of $\mathfrak g=T_eG$.
\end{thm}

As mentioned above, we will prove this well-known result as we could not find any simple direct proof in the literature.

\begin{proof} We first show that $T_e$ transforms a group 1-cocycle into a derivation. Then we explain why the $\R$-linear map $$T_e:Z^1_{\op{sm}}(G,\Ad)\to \op{Der}(\mathfrak{g})$$ is actually a bijection.

Let $\xi\in Z^1_{\op{sm}}(G,\Ad)$. Then $\xi\in\smoothF(G,\mathfrak{g})$ and, in view of the 1-cocycle condition, we have $\xi(e)=0$, $$\xi\circ L_g=\Ad_g\circ \xi+\xi(g)\;\;\text{and}\;\;\xi\circ R_h=\Ad_\bullet(\xi(h))+\xi\;.$$ When taking the derivative at $e$, we get \be\label{DE}T_g\xi\circ T_eL_g=\Ad_g\circ T_e\xi\;\;\text{and}\;\;T_h\xi\circ T_eR_h=\ad_\bullet(\xi(h))+T_e\xi\;.\ee Hence, $$\Ad_g\circ T_e\xi = T_g\xi\circ T_eR_g\circ T_gR_{g^{-1}}\circ T_eL_g=(\ad_\bullet(\xi(g))+T_e\xi)\circ\Ad_g\;.$$ If we set $\zd:=T_e\xi$ and if $Y\in\mathfrak{g}$, this equation reads $$\Ad_g(\zd(Y))=[\Ad_g(Y),\xi(g)]+\zd(\Ad_g(Y))\;.$$ It now suffices to derive the last identity (equality of functions in $\smoothF(G,\mathfrak{g})$) at $e$, and to evaluate the resulting identity (equality of linear maps in $\op{End}(\mathfrak{g})$) at $X\in\mathfrak{g}$. Indeed, we then obtain $$[X,\zd(Y)]=[[X,Y],\xi(e)]+[Y,\zd(X)]+\zd([X,Y])\;,$$ which is the desired result as $\xi(e)=0$.

Let us come to the second part and prove that $T_e$ is a bijection, i.e. that for any $\zd\in\Der(\mathfrak g)$ there is a unique $\xi\in Z^1_{\op{sm}}(G,\Ad)$ such that $T_e\xi=\zd$. Note that, in view of (\ref{DE}), if $\xi$ exists, it is a solution $\xi\in\smoothF(G,\mathfrak g)$ of the Cauchy problem \be\label{CP}T_g\xi=\Ad_g\circ\;\zd\circ (T_eL_g)^{-1}\;\;\text{and}\;\;\xi(e)=0\;.\ee Conversely, if $\xi$ is a solution, then $T_e\xi =\zd$ and $\xi$ is a 1-cocycle.

As for the cocycle property, observe that the coboundary operator is defined by $$d\xi(g,-)=\Ad_g\circ \xi - \xi\circ L_g + \xi(g)\in\smoothF(G,\mathfrak g)$$ and that the derivative $T_h$ of this map is given by $$T_h(d\xi(g,-))=\Ad_g\circ T_h\xi - T_{gh}\xi\circ T_hL_g\;.$$ If $\xi$ is, as assumed above, a solution of (\ref{CP}), this derivative vanishes. Indeed, $$T_h(d\xi(g,-))=\Ad_g\circ \Ad_h\circ\;\zd\circ (T_eL_h)^{-1} - \Ad_{gh}\circ\;\zd\circ (T_eL_{gh})^{-1}\circ T_hL_g=0\;,$$ since $\Ad$ is a group homomorphism. Since $G$ is connected, simply connected and $d\xi(g,e)=0$, it follows that $d\xi(g,h)=0$, for all $g,h\in G$.

It now suffices to show that (\ref{CP}) has a unique (global) solution $\xi\in\smoothF(G,\mathfrak g)$.

Just as the right Maurer-Cartan form $\zw^R$ is defined by $\zw^R_g=(T_eR_g)^{-1}\in \op{Iso}(T_gG,\mathfrak g)$, the left Maurer-Cartan form $\zw^L$ is given by $\zw^L_g=(T_eL_g)^{-1}\in\op{Iso}(T_gG,\mathfrak g)$. Viewed as a function of $g\in G$, the {\small RHS} $\Ad_g\circ\;\zd\circ\zw^L_g$ of the differential equation (\ref{CP}) is, just as the {\small LHS} $d_g\xi$, a 1-form in $\zW^1(G,\mathfrak g)$. We will show that the differential $d(\Ad_\bullet\circ\;\zd\circ\zw^L)$ of this 1-form vanishes. As any closed 1-form on a simply connected manifold is exact, it follows that there exists a function $\xi\in\smoothF(G,\mathfrak g)$ such that $d\xi=\Ad_\bullet\circ\;\zd\circ\zw^L$, i.e. that the Cauchy problem (\ref{CP}) admits a solution -- which is obviously unique.

It remains to prove that $d(\Ad_\bullet\circ\;\zd\circ\zw^L)=0$. Let $g\in G$ and $x,y\in T_gG$, denote by $x_e,y_e$ the corresponding vectors $\zw^L_gx,\zw^L_gy$ in $\mathfrak g$, and let $X=(\zw^L)^{-1}x_e, Y=(\zw^L)^{-1}y_e$ be the induced left invariant vector fields of $G$. It is enough to show that \be\label{Diff}d(\Ad_\bullet\circ\;\zd\circ\zw^L)(X,Y)(g)=0\;.\ee Indeed, this means that the value at $g$ of $d(\Ad_\bullet\circ\;\zd\circ\zw^L)$ vanishes on arbitrary vectors $x,y$.

The {\small LHS} of (\ref{Diff}) is the value at $g$ of $$X\cdot \Ad_\bullet(\zd y_e)-Y\cdot \Ad_\bullet(\zd x_e)-\Ad_\bullet(\zd[x_e,y_e])\;,$$ that is \be\label{Value}T_g(\Ad_\bullet(\zd y_e))(T_eL_g(x_e))-T_g(\Ad_\bullet(\zd x_e))(T_eL_g(y_e))-\Ad_g(\zd[x_e,y_e])\;.\ee The first term of (\ref{Value}) is the value at $x_e$ of the derivative at $e$ of $f=\Ad_\bullet(\zd y_e)\circ L_g$, i.e. of the function given by $$f(h)=\Ad_{gh}(\zd y_e)=(\Ad_g\circ \Ad_h)(\zd y_e)\;.$$ We get $$T_ef(x_e)=(\Ad_g\circ \ad_\bullet)(\zd y_e)(x_e)=\Ad_g[x_e,\zd y_e]\;.$$ Hence, we finally obtain $$\Ad_g[x_e,\zd y_e]-\Ad_g[y_e,\zd x_e]-\Ad_g(\zd[x_e,y_e])=\Ad_g([\zd x_e,y_e]+[x_e,\zd y_e]-\zd[x_e,y_e])=0\;.$$ This completes the proof. \end{proof}

\subsection{From {\small DLG}s to {\small DLA}s and vice versa}

If $(G,X)$ is a {\small DLG}, i.e. {\small LG} $G$ endowed with a multiplicative vector field $X$, then $(\mathfrak{g},\zd)$, with $$\mathfrak{g}=T_eG\;\;\text{and}\;\;\zd=T_e(\zw^R X)\;,$$ is a {\small DLA} -- the {\small DLA} of the {\small DLG} $(G,X)$. Conversely, if $(\mathfrak{g},\zd)$ is a {\small DLA}, then $(G,X)$, where $$G=\int \mathfrak g$$ is the unique simply connected {\small LG} integrating $\mathfrak g$ and where $$X=(T_e\circ\zw^R)^{-1}\zd\;,$$ is a simply connected {\small DLG} -- the unique simply connected {\small DLG} integrating $(\mathfrak{g},\zd)$.

We have thus proven the result announced in the beginning of this section:
\begin{thm} Any {\small DLG} differentiates to a {\small DLA}, and any {\small DLA} integrates to a unique simply connected {\small DLG}.\end{thm}

%%%%%%%%%%%%%%%%%%%%%%%%%%%%%%%%%%%%%%%%%%%%%%%%%%%%%%%%%%%%%
%%%%%%%%%%%%%%%%%%%%%%%%%%%%%%%%%%%%%%%%%%%%%%%%%%%%%%%%%%%%%
%%%%%%%%%%%%%%%%%%%%%%%%%%%%%%%%%%%%%%%%%%%%%%%%%%%%%%%%%%%%%

\bigskip

\newpage

\section{Graded Hopf algebras} \label{GLG}
In this section we present the construction of graded Hopf algebras -- the main
``tool'' for studying the {\small GLA} $\to$ {\small GLG} integration procedure.
 Then we discuss  multiplicative vector fields and the Maurer--Cartan formalism
in the context.

\subsection{Preliminaries}

\begin{defi}
A \define{graded manifold} is a paracompact Hausdorff unital-graded-algebra-ed space, locally modelled as
$\gradedF(U|V) \equiv \smoothF(U) \otimes SV$, where $U$ is an open subset of an $\RR^n$ and $V$ is a graded vector space with $V_0 = \{0\}$, and $SV$ is the graded
symmetric algebra on it.
\end{defi}

In the appendix \ref{sec:gr-man-app} we give  details related to this definition as well as describe the categorical properties of graded manifolds.
The appendix \ref{sec:funcan} is devoted to the properties of the functional space
$\gradedF({\cal M})$ of functions on a graded manifold ${\cal M}$.

Since the category of graded manifolds is cartesian monoidal, the following definition
is natural.

\begin{defi}[Graded Lie group]
The \define{category of graded Lie semigroups} (resp. \define{mo\-noids}) is the category of semigroup (resp. monoid) objects in the category of graded manifolds.
The \define{category of graded Lie groups} is the category of monoid objects in the category of graded manifolds which are groups.
\end{defi}

The following two results are straightforward as well and only quoted here for later reference.

\begin{lem}
Given linear maps between (unital) $R$-modules, where $R$ is a unital commutative ring\footnote{
From now on $R$ will be a field $k$ of characteristic $0$.
In this paper $k=\R$  unless the contrary is explicitly assumed, although for $k=\C$ apparently there is no conceptual issue either.},
$a \colon A \to B$,  $b \colon A' \to R$,  $c \colon B \to C$,
one has
$c \circ (a \otimes b) = (c \circ a) \otimes b \colon A \otimes A' \to C$.
Given
$a \colon A \to R$,  $b \colon A' \to B$,  $c \colon B \to C$,
one has
$c \circ (a \otimes b) = a \otimes (c \circ b) \colon A \otimes A' \to C$.
\begin{equation*}
\xymatrix{
A \ar[r]^a & B \ar@{}[d]_{\otimes} \ar@{.>}[r] & B \ar[r]^c &C\\
A' \ar[r]^b & R \ar@{.>}[ur]
}
\qquad,\qquad
\xymatrix{
A \ar[r]^a & B \ar@{}[d]_{\otimes} \ar@{.>}[dr]\\
A' \ar[r]^b & R \ar@{.>}[r] & B \ar[r]^c &C
}
\end{equation*}
\end{lem}

\begin{proof}
This is just a reformulation of the $R$-linearity of $c$ and the identification $B \otimes R \simeq B$.
\end{proof}

\begin{lem}
If $(A, \mu)$ is a unital graded commutative algebra, where $\mu$ is the multiplication, and $B$ is an $A$-bimodule and $X \in \calL(A,B)$ is a graded vector space morphism of degree $|X|$, then $X \in \Der(A,B)$ if and only if
\begin{equation}
X \circ \mu = \mu \circ (X \otimes \id +\id \otimes X)
\end{equation}
with implicit Koszul sign.
A derivation on a unital algebra vanishes on scalars.
\end{lem}

\begin{proof}
This is a reformulation of the Leibniz rule.
Let $f, g \in A$ with $f$ homogeneous.
Then $X \in \Der(A,B)$ means
$X(fg) = (Xf) g + (-1)^{|X||f|} f (Xg) = (X \otimes \id + \id \otimes X) (f \otimes g)$,
which means
$X \circ \mu (f \otimes g) = \mu \circ (X \otimes \id + \id \otimes X) (f \otimes g)$.
In particular, $1_A$ being of degree 0, one has $X(1_A^2) = (X1_A)1_A + (-1)^0 1_A(X1_A)$ so $X1_A=0$ and by linearity, $X$ vanishes on scalars.
\end{proof}

\subsection{Formulaire for graded Hopf algebras}

The monoidal product of the category of graded manifolds and the fact that their structure sheaves are Fr\'echet (see Appendix \ref{sec:funcan}) imply  that the structure sheaf of a graded Lie group is a sheaf of topological graded Hopf algebras.
In this subsection, we recall a few facts about these.
A good reference for the non-graded case is~\cite[Chapter III.1--III.3, 39--56]{kass}.

\begin{defi}

A topological \define{graded Hopf algebra} is a Fr\'echet graded vector space $H$ with structure maps:
\begin{enumerate}
\itemsep0em
\item
a multiplication $\mu \colon H \compTens H \to H$,
\item
a unit $\eta \colon \RR \to H$,
\item
a comultiplication $\Delta \colon H \to H \compTens H$,
\item
a counit $\epsilon \colon H \to \RR$,
\item
an antipode $S \colon H \to H$,
\end{enumerate}
satisfying the following axioms:
\begin{enumerate}
\itemsep0em
\item
unit laws
$\mu \circ (\id_H \otimes \eta) = \mu \circ (\eta \otimes \id_H) = \id_H$;
\item
associativity
$\mu \circ (\id_H \otimes \mu) = \mu \circ (\mu \otimes \id_H)$;
\item
counit laws
$(\id_H \otimes \epsilon) \circ \Delta = (\epsilon \otimes \id_H) \circ \Delta = \id_H$;
\item
coassociativity
$(\id_H \otimes \Delta) \circ \Delta = (\Delta \otimes \id_H) \circ \Delta$;
\item
the multiplication is a coalgebra morphism
$(\mu \otimes \mu) \circ (\id_H \otimes \tau \otimes \id_H) \circ (\Delta \otimes \Delta) = \Delta \circ \mu$, $\tau$ being the (involutive) flip operator, and
$\epsilon \otimes \epsilon = \epsilon \circ \mu$;
\item
the unit is a coalgebra morphism
$\Delta \circ \eta = \eta \otimes \eta$
and
$\epsilon \circ \eta = \id_{\RR}$;
\item
the antipode identity
$\mu \circ (\id_H \otimes S) \circ \Delta = \mu \circ (S \otimes\id_H) \circ \Delta = \eta \circ \epsilon$.
\end{enumerate}

\noindent By omitting the antipode structure together with the antipode identity we obtain the notion of a topological \define{unital and counital graded bialgebra}.
\noindent By dropping off the unit and the counit structures and the related identities, we obtain the general notion of a topological \define{graded bialgebra}.

\end{defi}
Note that the four morphic conditions are equivalent to saying that the comultiplication and the counit are algebra morphisms.
These maps should be continuous graded linear maps (of degree 0).
The completed tensor product is the projective one (to have a good representation of its elements as the absolutely convergent sums of decomposable tensors).
Generally, the Hopf algebra will be nuclear, so that the completed tensor product is well defined.

The antipode is an antimorphism.
In the commutative or cocommutative cases, it is bijective, so is an antiautomorphism; still under these conditions, it is an involution.
A morphism of Hopf algebras automatically intertwines the antipodes.

The example to keep in mind is  $H = \calC(G)$, with $\calC(G \times G) \simeq \calC(G) \compTens \calC(G)$.
In the non-graded case, if $G$ is a Lie semigroup (resp. monoid, group), then $\smoothF(G)$ is a topological unital  bialgebra (resp. unital and counital bialgebra,   Hopf algebra).

A note about notation: to keep in mind that a commutative Hopf algebra (unital and counital bialgebra, unital bialgebra) is thought of as a space of functions on a group (monoid, semigroup), respectively,
we use stars for the coalgebra maps, as if they were pullbacks, hence a counit $\epsilon = e^*$, comultiplication $\Delta = m^*$, and antipode (coinverse) $S = i^*$. In order to make computations more intuitive, later on we extend these notations to all Hopf algebras and bialgebras, even when they are not necessarily commutative.
When they are the unit, multiplication and inverse of a graded Lie group, we will drop the prefix ``co''.

We define the \define{constant map}
$\hat{e} = \eta \circ \epsilon = \eta \circ e^* \colon H \to H$.

In any Hopf algebra, $e^* \circ i^* = e^*$  and $i^* \circ \eta = \eta$ (\cite[Theorem~III.3.4.a~p.52]{kass}) --- this is part of the antipode being an antimorphism.
The other part reads: $i^*(fg) = i^*g i^*f$ and $m^* \circ i^* = (i^* \otimes i^*) \circ \tau \circ m^*$.

\begin{defi}
If $A$ is an algebra with multiplication $\mu$ and $C$ is a coalgebra with comultiplication $m^*$, and if $a,b \colon C \to A$
are linear maps, then we define their \define{convolution product} \\
$a \conv b = \mu \circ (a \otimes b) \circ m^* \colon C \to A$.
\end{defi}
One can show that $(\Hom(C,A), \conv, \hat{e})$ is an associative unital algebra (see~\cite[Proposition~III.3.1.a~p.50]{kass}).
%One has $(f \conv g)(x) = \sum f(x_{(1)}) g(x_{(2)})$.
The antipode identity then reads
\begin{eqnarray}
\label{antipode_convolution}\id \conv i^* = i^* \conv \id = \hat{e}
\end{eqnarray}
The identity (\ref{antipode_convolution}) implies that if an antipode exists (in a bialgebra), then it is unique.

\begin{coro}\label{cor:antipode_consequence}
If $a \in \Hom(H)$ such that $a \conv \id = 0$, then $a = 0$.
\end{coro}

\subsection{Left-invariant derivations of graded bialgebras}

\begin{defi}
\define{Vector fields} on graded manifolds are  derivations of the structure sheaf,
\begin{equation}
\vf(\calM) = \Der \left( \calO_\calM \right).
\end{equation}
\end{defi}

As in the non-graded case, if $f \colon \calM \to \calN$ is a smooth\footnote{Talking about smoothness in the graded setting is obviously a language abuse, what is meant is the class of functions $f$ with a smooth body part $\tilde f$ and the appropriate graded part. To keep the intuition from the non-graded case we will however write
$\smoothF(\cdot)$ meaning $\gradedF(\cdot)$} map, and $X \in \vf(\calM)$ and $Y \in \vf(\calN)$ are vector fields, then $X$ and $Y$ are \define{$f$-related}, which we denote by $X \sim_f Y$, if $Tf \circ X = Y \circ f$.
If $X \sim_f Y$ and $Y \sim_g Z$, then $X \sim_{g \circ f} Z$, and $X \sim_{\id_M} X$, so graded manifolds with a vector field and smooth maps relating them form a category.
This category is cartesian with product $(\calM,X) \times (\calN,Y) = (\calM \times \calN, X \otimes \id_{\calO_\calN} + \id_{\calO_\calM} \otimes Y)$ and obvious projections, and terminal object $((\{\ast\},0),0)$.

In the non-graded case, $X \in \linvvf(G)$ means that for all $g \in G$, $X$ is $L_g$-related to itself, that is, $X \circ L_g = TL_g \circ X$.
Seeing $X$ as a derivation of $\smoothF(G)$, this means that for all $f \in \smoothF(G)$ and $g \in G$, one has
$(Xf) \circ L_g = Tf \circ X \circ L_g = Tf \circ TL_g \circ X = T(f \circ L_g) \circ X = X(f \circ L_g)$.

We want to express this in terms of the Hopf algebra $\smoothF(G)$:
\begin{equation}
(\eval_g \otimes \id) \circ m^* f = f \circ L_g
\qquad\text{and}\qquad
(\id \otimes \eval_g) \circ m^* f = f \circ R_g.
\end{equation}
This can also be obtained by noting that $L_g = m \circ (\dot{g} \otimes \id)$, where $\dot{g}$ is the constant $g$, so by dualizing, $L_g^* = (\eval_g \otimes \id) \circ m^*$.

Using this to translate the condition of left-invariance $(Xf) \circ L_g = X(f \circ L_g)$ gives
$(\eval_g \otimes \id) \circ m^*Xf = X (\eval_g \otimes \id) \circ m^*f$ for all $f$ and $g$, so
$(\eval_g \otimes \id) \circ m^*X = X (\eval_g \otimes \id) \circ m^*= (\eval_g \otimes \id) \circ (\id \otimes X) \circ m^*$ for all $g$. Therefore $X$ is left-invariant if and only if
$(\id \otimes X) \circ m^* = m^* \circ X$,
and this is taken as the definition in the case of a graded Lie semigroup (and, furthermore, of a graded bialgebra as soon as we replace $m^*$ with a general comultiplication):

\begin{defi}%[Left-invariant vector field]
\label{defi:linder} Let $(B,\mu, m^*)$ be a graded bialgebra. A \define{left-invariant derivation} of $B$
is  an element of
\begin{equation}
\label{eq:left-derivation}
 \Der^{\mathrm{left}}(B) = \left\{
X \in \Der (B) \mid
(\id \otimes X) \circ m^* =m^* \circ X
\right\}
\end{equation}
Here $\Der (B)$ is the graded Lie algebra of graded derivations of $B$ regarded as an algebra.
Similarly, $X$ is \define{right-invariant derivation} if and only if $(X \otimes \id) \circ m^* = m^* \circ X$.\\
\end{defi}

\begin{prop}\label{prop:linvder_is_algebra} The space of left- (right-)invariant derivations of a graded bialgebra is closed under the graded Lie bracket.
\end{prop}
\begin{proof} Follows immediately from the definition of left- (right-)invariant derivations.
\end{proof}

\begin{defi}%[Left-invariant vector field]
\label{defi:linvvf} A \define{left-invariant vector field} on a graded Lie semigroup $G$ is a left-invariant derivation of the corresponding graded unital bialgebra of functions, i.e.
an element of
\begin{equation}
\label{eq:linvvf}
\linvvf(G) = \left\{
X \in \vf(G) \mid
(\id \otimes X) \circ m^* = m^* \circ X
\right\}
\end{equation}
where $m$ denotes the multiplication of $G$.
Similarly, $X$ is \define{right-invariant} if and only if $(X \otimes \id) \circ m^* = m^* \circ X$.\\
\end{defi}
\begin{rmk}
Although the multiplication is not written in the above definition, the bialgebra structure is needed: it is ``hidden'' in $X \in \vf(G)$.
\end{rmk}

\subsection{The graded Lie algebra of a graded Lie group}

To obtain the notion of tangent vector, we introduce the following:

\begin{defi}
If $(A,\mu)$ is a graded commutative $R$-algebra, $a \in A'$ is a linear form, and $B$ is an $A$-bimodule, an \define{$a$-derivation} from $A$ to $B$ is an element of
\begin{equation}
\Der_a(A,B) = \left\{
\xi \in \calL(A,B) \mid
\xi \circ \mu = \xi \otimes a + a \otimes \xi
\right\}.
\end{equation}
%If $a$ is the evaluation at $x$, we write $\Der_x$ for $\Der_a$.
\end{defi}

In a graded Lie monoid with unit $e^*$, we write $\Der_e$ for $\Der_{e^*}$.

If $X \in \vf(G)$, we write, by abuse of notation and analogy with evaluations ``$e^* = \eval_e$'', $X_e = e^* \circ X \in \Der_e(\gradedF(G),\RR) = \Der(\gradedF(G)_e,\RR)$. %(``germ-determination'' of derivations).

The following proposition is the exact analogue of~\cite[Proposition~7.2.3~p.115]{ccf}.

\begin{prop}\label{prop:left_translation}
If $G$ is a graded Lie monoid, there exists a graded linear isomorphism (of degree 0)
\begin{align*}
\linvvf(G) & \longisomto T_e G = \Der(\calC(G)_e,\RR)\\
X & \longmapsto X_e\\
\id\conv v  & \longmapsfrom v.
\end{align*}
In particular, a left-invariant vector field is determined by its ``value at the unit''.
\end{prop}

\begin{proof}
First, we verify that $X=\id\conv v$ is a derivation. Indeed, by the compatibility of $\mu$ and $m^*$
$$
X\circ \mu = \mu\circ (\id\otimes v)\circ m^*\circ \mu =
\mu\circ (\id\otimes v)\circ (\mu\otimes \mu)\circ (\id\otimes \tau \otimes \id)\circ (m^*\otimes m^*)
$$
Using that $v$ is a derivation at $e^*$, we obtain
$$
 (\id\otimes v)\circ (\mu\otimes \mu) =(\mu\otimes \mu)\circ
(\id\otimes\id\otimes e^*\otimes v +\id\otimes\id\otimes v\otimes e^*)
$$
Therefore
$$
(\id\otimes v)\circ (\mu\otimes \mu)\circ (\id\otimes \tau \otimes \id)= (\mu\otimes \mu)\circ (\id\otimes \tau \otimes \id)
\circ (\id\otimes e^*\otimes \id \otimes v +\id\otimes v\otimes\id\otimes e^*)
$$
Taking into account that $(\id\otimes v)\circ m^* = X\otimes 1$ and $(\id\otimes e^*)\circ m^*= \id\otimes 1$, we get
\begin{eqnarray}\nonumber
X\circ \mu = \mu\circ (\mu\otimes \mu)\circ (\id\otimes \tau \otimes \id)\circ (\id\otimes 1\otimes X\otimes 1 + X\otimes 1\otimes \id\otimes 1)
=\\ \nonumber
\mu\circ (\mu\otimes \mu)\circ (\id\otimes X\otimes 1\otimes 1 + X\otimes \id\otimes 1\otimes 1)=
\mu\circ (\id\otimes X + X\otimes \id)
\end{eqnarray}

\noindent Second, we check that if $v \in T_e G$, then $\id\conv v=\mu\circ (\id \otimes v) \circ m^*$ is indeed left-invariant. One has
$$
(\id\otimes X)\circ m^* = (\id\otimes \mu)\circ (\id\otimes\id\otimes v)\circ (\id\otimes m^*)\circ m^*
$$ Using the coassoativity condition, we obtain
$$
(\id\otimes X)\circ m^* = (\id\otimes \mu)\circ (\id\otimes\id\otimes v)\circ (m^*\otimes \id)\circ m^*=
 (\id\otimes \mu)\circ  (m^*\otimes \id)\circ (\id\otimes v)\circ m^*
$$
Taking into account that $v$ takes values in constants, one can derive the following two identities (at the moment we need the first one):
\begin{eqnarray}\label{first_mu-m-v}  (\id\otimes \mu)\circ  (m^*\otimes \id)\circ (\id\otimes v) &=&
m^*\circ \mu\circ (\id\otimes v) \\ \label{second_mu-m-v}
(\mu\otimes \id)\circ  (\id\otimes m^*)\circ (v\otimes \id) &=&
m^*\circ \mu\circ (v\otimes\id)
\end{eqnarray}
Therefore by (\ref{first_mu-m-v}) we get the desired left-invariance of $X$:
$$
(\id\otimes X)\circ m^* = m^*\circ \mu\circ (\id\otimes v)\circ m^* =m^*\circ X
$$

\noindent Third, postcomposing the left-invariance relation with $\mu\circ(\id \otimes e^*)$ gives
$$
\id\conv X_e=\mu\circ (\id \otimes X_e) \circ m^* = \mu\circ (\id \otimes e^*) \circ m^* \circ X = X
$$
by the counit law.
%Postcomposing the left-invariance relation with $e^* \otimes \id$ gives $(e^* \otimes X) \circ m^* = X$, that is, $f_{(1)}^i(e) Xf_{(2)}^i(x) = Xf(x)$ which is the case $g=e$ of the above relation $X(f \circ L_g) = f_{(1)}^i(g) Xf_{(2)}^i$.
Lastly, we check, although this is not necessary in finite dimension, that $v =  e^*\circ\mu \circ (\id \otimes v) \circ m^*$.
We have $e^* \circ\mu\circ (\id \otimes v ) \circ m^* =\mu\circ (e^*\otimes e^*)\circ (\id \otimes v ) \circ m^*$  from the unit law.
From $e^*\circ v=v$ we immediately get
$e^*\circ (\id\conv v)=\mu\circ (e^*\otimes e^*\circ v)\circ m^* =\mu\circ (\id\otimes v)\circ (e^*\otimes \id)\circ m^*$. One the other hand
the counit property implies $ (e^*\otimes \id)\circ m^* =1\otimes \id \simeq \id$, therefore
$e^*\circ (\id\conv v)=\mu\circ (\id\otimes v)\circ (1\otimes \id) =v$.
\end{proof}

\begin{coro} Defined via the previous proposition, $T_e G$ is a graded Lie algebra. This fact is verified using the Proposition \ref{prop:linvder_is_algebra}
and Proposition \ref{prop:left_translation}: the first one implies that the space $\linvvf(G)$ is a graded Lie subalgebra of $\vf(G)$,
 while the second one gives us an explicit isomorphism between $\linvvf(G)$ and $T_e G$.
\end{coro}

Along with left translations $v^L\colon =\id\conv v$, we define right translations $v^R\colon = v\conv \id$, which are also derivations of the multiplication
$\mu$ (the proof is similar to the proof of  Proposition \ref{prop:left_translation}).

\begin{prop}\label{prop:left-right_commute} For any $v_1$ and $v_2$ the corresponding left and right translations super commute, i.e.
$v_1^L\circ v_2^R=(-1)^{k_1 k_2}v_2^R\circ v_1^L$,
where $k_1$ and $k_2$ are the degrees of $v_1$ and $v_2$, respectively.
\end{prop}
\begin{proof}
Taking into the account that $v_1^L$ is a derivation and that, in particular, $v_1^L$ annihilates constants, one has
$$
v_1^L\circ v_2^R =v_1^L\circ \mu\circ ( v_2\otimes \id)\circ m^*
=\mu\circ (v_1^L\otimes \id +\id\otimes v_1^L)\circ ( v_2\otimes \id)\circ m^*=
\mu\circ (\id\otimes v_1^L)\circ ( v_2\otimes \id)\circ m^*
$$
From the commutation relation $ (\id\otimes v_1^L)\circ ( v_2\otimes \id)=(-1)^{k_1k_2}( v_2\otimes \id)\circ (\id\otimes v_1^L)$
we immediately obtain
$v_1^L\circ v_2^R =(-1)^{k_1k_2} \mu\circ ( v_2\otimes \id)\circ (\id\otimes v_1^L) \circ m^*$. On the other hand, by definition
$v_1^L$ is left-invariant, therefore
$$
(\id\otimes v_1^L) \circ m^* = m^*\circ v_1^L\,.
$$
and finally,
$$
v_1^L\circ v_2^R =(-1)^{k_1k_2} \mu\circ ( v_2\otimes \id)\circ m^*\circ v_1^L  =(-1)^{k_1k_2} v_2^R \circ v_1^L
$$

\end{proof}

\begin{rmk} As we never used the commutativity assumption in this section, the statements of Proposition \ref{prop:left_translation}
and Proposition \ref{prop:left-right_commute} remain true for arbitrary graded unital counital bialgebras.
\end{rmk}

\subsection{Multiplicative vector fields on graded Lie groups}

In the non-graded case, $X \in \multvf(G)$ means that $(X,X)$ and $X$ are $m$-related, meaning $X \circ m = Tm \circ (X,X)
 = Tm \circ (X \otimes \id + \id \otimes X)$ by the Leibniz rule.
$X_{gh} = T_hL_g (X_h) + T_gR_h (X_g)$ for all $g,h \in G$.
In terms of derivations, this means
$(Xf) \circ m = Tf \circ X \circ m = Tf \circ Tm \circ (X \otimes \id + \id \otimes X) = T(f \circ m) \circ (X \otimes \id + \id \otimes X) = (X \otimes \id + \id \otimes X) (f \circ m)$.
Therefore, $X$ is multiplicative if and only if
$(X \otimes \id + \id \otimes X) \circ m^* = m^* \circ X$, and this is taken as the definition in the graded case:

\begin{defi}%[Multiplicative vector field]
\label{defi:multvf1}
A \define{multiplicative vector field} on a graded Lie semigroup $G$ is an element of
\begin{equation}
\label{eq:multvf2}
\multvf(G) = \left\{
X \in \vf(G) \mid
(X \otimes \id + \id \otimes X) \circ m^* = m^* \circ X
\right\}
\end{equation}
where $m$ denotes the multiplication of $G$.
This means exactly that $(X,X) \sim_{m^*} X$, that is, $X$ is compatible with the multiplication. \\
Similarly, we say that for a graded Lie monoid, $X$ is compatible with the unit if $X \sim_{e^*} 0$, \\ that is, $e^* \circ X = X_e = 0$; \\
and for a graded Lie group  $X$ is compatible with the inverse if $X \sim_{i^*} X$.
\end{defi}

The following  is the graded analogue of Proposition~\ref{prop:multComp}.
\begin{prop}
For a graded Lie monoid, a multiplicative vector field $X$ is compatible with the unit.
\begin{equation}
X_e := e^* \circ X = 0
\end{equation}
For a graded Lie group, a multiplicative vector field is compatible with the inverse.
\begin{equation}
i^* \circ X = X \circ i^*.
\end{equation}
\end{prop}

\begin{proof}
Postcomposing the multiplicativity relation with $\id \otimes e^*$ gives $(X \otimes e^* + \id \otimes X_e ) \circ m^* = X$.
Since $X \otimes e^* = X \circ (\id \otimes e^*)$ and $(\id \otimes e^*) \circ m^* = \id$, the left-hand side above is equal to $X + (\id \otimes X_e ) \circ m^*$.
Therefore $ (\id \otimes X_e ) \circ m^* = 0$ so $\id \conv X_e = 0$ so $X_e = 0$ by  Corollary \ref{cor:antipode_consequence} (or directly postcomposing with $e^* \otimes e^*$).

For the inverse,
the (right) antipode identity reads $\id \conv i^* = \mu \circ (\id \otimes i^*) \circ m^* = \hat{e} = \eta \circ e^*$.
Postcomposing it with a vector field $X$ gives us
$$
X \circ \eta \circ e^*=X \circ \mu \circ (\id \otimes i^*) \circ m^*
= \mu \circ (X \otimes \id + \id \otimes X) \circ (\id \otimes i^*) \circ m^*
$$
The left-hand side of the above formula equals to $0$ since a derivation vanishes on scalars. Therefore
$$
0= \mu \circ (X \otimes i^* + \id \otimes (X \circ i^*)) \circ m^*=\mu \circ
(\id\otimes i^*)\circ (X\otimes\id)\circ m^* + \id\conv (X\circ i^*)
%= X \conv i^* + (-1)^{...} \id \conv (X \circ i^*)
$$
Taking into account that $X$ is multiplicative we have $(X\otimes\id)\circ m^* =m^*\circ X-(\id\otimes X)\circ m^*$. This immediately implies
$$
\mu \circ (\id\otimes i^*)\circ (X\otimes\id)\circ m^*=(\id\conv i^*) \circ X- \id\conv (i^*\circ X)
$$
But $\id\conv i^*=\hat e=\eta\circ e^*$. Given that $X_e=e^*\circ X=0$ we obtain $(\id\conv i^*) \circ X$, thus $\mu \circ (\id\otimes i^*)\circ (X\otimes\id)\circ m^*=- \id\conv (i^*\circ X)$ and $0=\id\conv \left(X\circ i^*-i^*\circ X\right)$
By the consequence of the antipode identity (Corollary \ref{cor:antipode_consequence}), this implies $X \circ i^* = i^* \circ X$ as wanted.
\end{proof}

\begin{coro}
The category of graded Lie semigroups (resp. monoids, groups) with a multiplicative vector field is isomorphic to the category of semigroup (resp. monoid, group) objects in the category of graded manifolds with a vector field, and maps preserving them.
\end{coro}

\begin{prop}\label{prop:exact} Let $G$ be a graded monoid, $H$ be its bialgebra of functions, and $v$ be a derivation at $e$.
Then $X=\id\conv v-v\conv\id$ is a multiplicative vector field.
\end{prop}
\begin{proof} We have
\begin{eqnarray}\nonumber
(X\otimes \id)\circ m^* &=& (\mu\otimes\id)\circ \left( (\id\otimes v-v\otimes\id)\otimes\id \right)\circ (m^*\otimes \id)\circ m^*
\\ \nonumber
(\id\otimes X)\circ m^* &=& (\id\otimes\mu)\circ \left(\id\otimes (\id\otimes v-v\otimes\id) \right)\circ (\id\otimes m^*)\circ m^*
\end{eqnarray}
Thus $(X\otimes \id +\id\otimes X)\circ m^*= (I)+(II)$, where
\begin{eqnarray}\nonumber
(I) &=& (\mu\otimes\id)\circ (\id\otimes v\otimes \id)\circ  (m^*\otimes \id)\circ m^*-
 (\id\otimes\mu)\circ (\id\otimes v\otimes \id)\circ  (\id\otimes m^*)\circ m^*
\\ \nonumber
(II) &=&   (\id\otimes\mu)\circ (\id\otimes\id\otimes v)\circ (\id\otimes m^*)\circ m^*-
(\mu\otimes\id)\circ (v\otimes\id\otimes\id)\circ (\otimes m^*\otimes \id)\circ m^*
\end{eqnarray}
Thanks to the coassociativity law and the identities (\ref{first_mu-m-v}) and (\ref{second_mu-m-v}),
the first term $(I)$ vanishes, while the second term equals to
\begin{eqnarray}\nonumber
(II)=(\id\otimes \mu)\circ (m^*\otimes\id)\circ (\id\otimes v)\circ m^* -
(\mu\otimes \id)\circ (\id\otimes m^*)\circ (v\otimes \id)\circ m^*=\\ \nonumber =
m^*\circ \mu\circ (\id\otimes v-v\otimes \id)\circ m^* = m^*\circ X
\end{eqnarray}
Finally $(X\otimes \id +\id\otimes X)\circ m^*= m^*\circ X $ which proves that $X$ is multiplicative.
\end{proof}

%\label{first_mu-m-v}

\begin{rmk} Although in this section we deal with commutative bialgebras representing functions on graded monoids, we do not use commutativity in the proofs, therefore all statements, like in the previous subsection,  are also valid in the general (non-commutative) case. In the next subsection, however, the commutativity will be important.
\end{rmk}

\subsection{The Maurer--Cartan automorphism of a graded Lie group}
\label{sec:mc}

\begin{defi}
The right (resp. left) \define{Maurer--Cartan automorphism} of a graded Lie group $\calG$ is given by
\begin{equation}
\omega^R = \bullet \conv i^*
\end{equation}
(resp. $\omega^L = i^* \conv \bullet$).
\end{defi}

By the hexagon identity, $\left( \omega^R \right)^{-1} = \bullet \conv \id$ and $\left( \omega^L \right)^{-1} = \id \conv \, \bullet$ and $\omega^R$ and $\omega^L$ are linear automorphisms of $\calL(H)$.

In the non-graded case, this coincides with the usual definition of what we called $\omega^R_\bullet$.
It had values in $\smoothF(G,\mathfrakg)$, which should therefore be replaced by $\smoothF(G,\Der_e(\calC_G,\RR))$, morphism of graded rings.

Note also that there is no inclusion between $\Der_e(\calC_G,\calC_G)$ and $\Der(\calC_G,\calC_G)$.

To relate this to the classical case, we prove the following proposition.

\begin{prop}
If $(G,e)$ is a Lie monoid over $k$, then the map
\begin{align*}
\Psi \colon \smoothF(G,\mathfrakg) &\longisomto \Der_e(\smoothF(G), \smoothF(G))\\
\xi & \longmapsto \big( f \mapsto \xi(\bullet) \cdot f \big)
\end{align*}
is a linear isomorphism with inverse given by
$\Psi^{-1}(u) \colon x \mapsto \big( f \mapsto u(f)(x) \big)$
where we used the isomorphism between $\mathfrakg = T_eG$ and  $\Der_e(\smoothF(G),k)$.
\end{prop}

\begin{proof}
If $\xi \in \smoothF(G,\mathfrakg)$ and $f, g \in \smoothF(G)$, then $\Psi(\xi)(gh) \colon x \mapsto \xi(x) \cdot fg = (\xi(x) \cdot f)g(e) + f(e) (\xi(x) \cdot g)$, so $\Psi(\xi) \in \Der_e(\smoothF(G),\smoothF(G))$. \\
Conversely, if $u \in \Der_e(\smoothF(G),\smoothF(G))$ and $x \in G$ and $f, g \in \smoothF(G)$, then
$\Psi^{-1}(u)(x) (fg) = u(fg)(x) = u(f)(x) g(e) + f(e) u(g)(x)$ so $\Psi^{-1}(u)(x) \in \mathfrakg$, and $\Psi^{-1}(u)$ is smooth;
the latter can be verified by standard technique using a partition of unity.
\end{proof}

\begin{prop}\label{prop:MC}
The Maurer--Cartan automorphism restricts to the linear isomorphism
\begin{equation}
\omega^R \colon \vf(G) \longisomto \Der_e(\calC_G,\calC_G).
\end{equation}
\end{prop}

\begin{lem}
A special instance of the $(\otimes, \circ)$-interchange identity is:
if $a \in \calL (A,A')$ and $b \in \calL (B,B')$, then $a \otimes b = (\id_{A'} \otimes b) \circ (a \otimes \id_{B}) = (a \otimes \id_{B'}) \circ (\id_{A} \otimes b) \colon A \otimes B \to A' \otimes B'$.
Together with the unit law, if $a \colon A \to H$, this gives $\mu \circ (a \otimes \eta) = \mu \circ (\eta \otimes a) = a \colon A \to H$.
In particular (using again the interchange property), if $a \in \calL(A,H)$, then
$\mu \circ (a \otimes \hat{e}) = a \otimes e^* \colon A \otimes H \to H$.
Together with the counit law, if $a \colon H \to A$, then $(a \otimes \hat{e}) \circ \Delta = a \otimes \eta \colon H \to A \otimes H$.
\end{lem}

\begin{proof}
Straightforward.
\end{proof}

\begin{proof}[Proof of the proposition]
Suppose that $X \circ \mu = \mu \circ (X \otimes \id + \id \otimes X)$.
Then
\begin{align*}
(X \conv i^*) \circ \mu
&= \mu \circ (X \otimes i^*) \circ m^* \circ \mu\\
&= \mu \circ (X \otimes i^*) \circ (\mu \otimes \mu) \circ \tau_{1324} \circ (m^* \otimes m^*)\\
&= \mu \circ \big( (X \circ \mu) \otimes (i^* \circ \mu) \big) \circ \tau_{1324} \circ (m^* \otimes m^*)\\
&= \mu \circ \big( ( \mu \circ (X \otimes \id + \id \otimes X)) \otimes (\mu \circ (i^* \otimes i^*) \circ \tau) \big) \circ \tau_{1324} \circ (m^* \otimes m^*)\\
&= \mu \circ (\mu \otimes \mu) \circ \big( (X \otimes \id + \id \otimes X) \otimes (i^* \otimes i^*) \big) \circ \tau_{1342} \circ (m^* \otimes m^*)
\end{align*}
Taking into account that $\mu$ is commutative\footnote{This is one of few cases where commutativity is needed.},
i.e. $\mu\circ\tau=\mu$ and thus $\mu \circ (\mu \otimes \mu)\circ \sigma =\mu \circ (\mu \otimes \mu)$
for any permutation $\sigma\in \mathbb{S}_4$, we obtain
\begin{align*}
(X \conv i^*) \circ \mu
&= \mu \circ ( X \conv i^* \otimes \id \conv i^* + \id \conv i^* \otimes X \conv i^*)\\
&= \mu \circ ( X \conv i^* \otimes \hat{e} + \hat{e} \otimes X \conv i^*)= X \conv i^* \otimes e^* + e^* \otimes X \conv i^*.
\end{align*}
Conversely, suppose that $\xi \circ \mu = \xi \otimes e^* + e^* \otimes \xi$.
Then
\begin{align*}
(\xi \conv \id) \circ \mu
&= \mu \circ (\xi \otimes \id) \circ m^* \circ \mu\\
&= \mu \circ (\xi \otimes \id) \circ (\mu \otimes \mu) \circ \tau_{1324} \circ (m^* \otimes m^*)\\
&= \mu \circ \big( (\xi \circ \mu) \otimes \mu \big) \circ \tau_{1324} \circ (m^* \otimes m^*)\\
&= \mu \circ \big( (\xi \otimes e^* + e^* \otimes \xi) \otimes \mu \big) \circ \tau_{1324} \circ (m^* \otimes m^*)\\
&= \mu \circ \big( \mu \circ (\xi \otimes e^* \otimes \id + e^* \otimes \xi \otimes \id) \otimes \id \big) \circ \tau_{1324} \circ (m^* \otimes m^*)\\
&= \mu \circ \big( \mu \circ (\xi \otimes \id \otimes e^* + e^* \otimes \id \otimes \xi) \otimes \id \big) \circ (m^* \otimes m^*)\\
&= \mu \circ \big( ( (\mu \circ (\xi \otimes \id)) \otimes e^* + e^* \otimes (\mu \circ (\id \otimes \xi)) ) \otimes \id \big) \circ (m^* \otimes m^*)\\
&= \mu \circ \big( (\mu \circ (\xi \otimes \id)) \otimes e^* \otimes \id + e^* \otimes \id \otimes (\mu \circ (\xi \otimes \id)) \big) \circ (m^* \otimes m^*)\\
&= \mu \circ (\xi \conv \id \otimes \id + \id \otimes \xi \conv \id)
\end{align*}
as wanted.
\end{proof}

%The Maurer--Cartan isomorphism restricts to the linear isomorphism $\omega^R \colon \linvvf(\calG) \longisomto \left\{ \xi \in \Der_e(\calC_\calG,\calC_\calG) \mid m^* \circ \xi = \xi \otimes 1 \right\} \simeq \mathfrakg$...

\begin{defi}
The \define{adjoint action} $\Ad \colon H \mapsto H \otimes H$ of a Hopf algebra $H$ by
\begin{equation}
\Ad =  (\mu \otimes \id) \circ (\id \otimes \id \otimes i^*) \circ {m^*}^{(3)}
\end{equation}
and the set of 1-cocycles of a graded Lie group $\calG$ by
\begin{equation}
Z^1(G,\Ad) = \left\{ \xi \in \Der_e(\calC_G,\calC_G) \mid m^* \circ \xi = \xi \otimes 1 + (\xi \otimes \id) \circ \tau \circ \Ad \right\}.
\end{equation}
\end{defi}
%example 2.7 , Majid_Algebras and Hopf algebras in braided categories and Appendix of Majid_Quantum and braided linear algebra

\begin{prop}\label{prop:MC-1-isomorphism}
The Maurer--Cartan isomorphism restricts to the linear isomorphism
\begin{equation}
\omega^R \colon \multvf(\calG) \longisomto Z^1(\calG,\Ad)
\end{equation}
\end{prop}

\begin{proof}
Prop. \ref{prop:MC-1-isomorphism} is a straightforward analogue of Prop. \ref{prop:MCiso},
phrased in terms of convolution products. We shall prove it using a more general fact.

%\vskip 2mm
\noindent Let $V$ be an $H-$bicomodule
with left and right coactions $\rho_V^L\colon V\to H\otimes V$ and $\rho_V^R\colon V\to V\otimes H$, respectively,
and let $C^n (V,H)=\Hom (V,H^{\otimes n})$ with coface operators $\delta_i\colon C^n (V,H)\to C^{n+1} (V,H)$, such that
\begin{eqnarray}\label{eqn:coface_i}
\delta_i (c)=
\left\{
\begin{array}{cc}
(\id \otimes c)\circ \rho_V^L\,, &  i=0 \\
\left(\id^{\otimes (i-1)}\otimes m^* \otimes\id^{\otimes (n-i)}\right)\circ c \,, & 1\le i\le n\\
(c\otimes\id)\circ \rho_V^R \, , & i=n+1
\end{array}
\right.
\end{eqnarray}
The alternate sum of $\delta_i$
\begin{equation}
\delta=\sum\limits_{i=0}^{n+1} (-1)^i\delta_i
\end{equation}
is a nilpotent operator, i.e. $\delta^2=0$.
Using that $\mu\colon H\otimes H\to H$ is a morphism of bialgebras and the antipode map $i^*$ is an anti-comorphism, i.e. $m^*\circ i^* = i^*\otimes i^* \circ\tau\circ m^*$,
we construct a new left $H-$comodule structure on $V$:
\begin{equation}\label{eqn:new_rho}
\rho_V^{new} = (\mu\otimes\id)\circ(\id\otimes\tau)\circ\left(\id^{\otimes 2}\otimes i^*\right)\circ\left(\rho_V^L\otimes\id\right)\circ \rho_V^R
\end{equation}
Now $V$ has a new bicomodule structure, where the left coaction is given by $\rho_V^{new}$, while the right comodule structure is the trivial one $\id\otimes 1\colon V\to V\otimes H$.
Therefore there exist new $\delta_i^{new}$ and the differential $\delta^{new}=\sum\limits_{i=0}^{n+1} (-1)^i\delta_i^{new}$. Define
$\omega^R_n\colon C^n (V,H)\to C^n (V,H)$ by the formula
\begin{equation}\label{eqn:general_MC}
\omega^R_n (c) = \mu_{H^{\otimes n}}\circ \left(c\otimes (m^*)^{n}\circ i^*\right)\circ\rho_V^R\,,
\end{equation}
where
\begin{eqnarray}
(m^*)^{n} =\left\{
\begin{array}{cc}
\id \, , & n=1 \\
\left( m^*\otimes \id^{\otimes (n-2)}\right) \circ\ldots\circ \left( m^*\otimes \id\right)\circ  \circ m^*  \, , & n\ge 2
\end{array}
\right.
\end{eqnarray}
and
\begin{equation}
\mu_{H^{\otimes n}}=\mu^{\otimes n}\circ\tau_{1n+1 \ldots n2n}\colon H^{\otimes n}\otimes H^{\otimes n}\to  H^{\otimes n}
\end{equation}
is the canonical extension of the multiplication
$\mu$ to the $n-$the tensor power of $H$.

\begin{lem}\label{lem:MC_interchange} One has for all $n$ and $i=0,\ldots, n+1$
\begin{equation}\label{eqn:MC_interchange}
\omega^R_{n+1}\circ \delta_i = \delta^{new}_i \omega^R_{n}
\end{equation}
\end{lem}
\begin{proof} The proof is canonical and straightforward. To make it more intuitive and visual, we "dualize" the picture by considering of $H-$(bi)modules instead of $H-$(bi)comodules,
$C^n(H,V)=\Hom (H^{\otimes n}, V)$ instead of $C^n(V,H)$ and by assuming that $H$ is non-graded. We denote $ab=\mu (a,b)$, $\rho_V^L(a,v)=av$ and
$\rho_V^R(v,a)=va$ for $a,b\in H, v\in V$, where $\rho_V^L \colon H\otimes V\to V$ and $\rho_V^R \colon V\otimes H\to H$ are the left- and right- $H-$module
structures on $V$, respectively.  Now the dual analogue of (\ref{eqn:coface_i}) is
\begin{eqnarray}
\delta_i (c)=
\left\{
\begin{array}{cc}
\rho_V^L\circ (\id \otimes c)\,, &  i=0 \\
c\circ \left(\id^{\otimes (i-1)}\otimes\mu \otimes\id^{\otimes (n-i)}\right) \,, & 1\le i\le n\\
\rho_V^R\circ (c\otimes\id) \, , & i=n+1
\end{array}
\right.
\end{eqnarray}
or, more explicitly,
\begin{eqnarray}
\delta_i (c) (a_1, \ldots a_n) =
\left\{
\begin{array}{cc}
a_1 c(a_2, \ldots , a_n+1)\,, &  i=0 \\
c(\ldots, a_i a_{i+1}, \ldots) \,, & 1\le i\le n+1\\
c(a_1, \ldots , a_{n})a_{n+1} \, , & i=n+1
\end{array}
\right.
\end{eqnarray}
for all $a_1, \ldots, a_{n+1}\in H$. By use of the Sweedler notation (cf.~\cite{kass})
$m^* (a)=\sum a'\otimes a''$ we rewrite the dual analogue of \ref{eqn:new_rho}
\begin{equation}
\rho_V^{new} =\rho_V^R\circ \left(\rho_V^L\otimes\id\right)
\circ\left(\id^{\otimes 2}\otimes i^*\right)
\circ(\id\otimes\tau)\circ
 (m^*\otimes\id)
\end{equation}
as follows:
$$
\rho_V^{new} (a,v)=\sum a'v\, i^*(a'') \, .
$$
Likewise, the dual analogue of \ref{eqn:general_MC}
\begin{equation}
\omega^R_n (c) =\rho_V^R\circ \left(c\otimes  i^*\circ (\mu)^{n}\right) \circ (m^*)_{H^{\otimes n}}
\end{equation}
where
\begin{eqnarray}
\mu^{n} =\left\{
\begin{array}{cc}
\id \, , & n=1 \\
\mu\circ \left( \mu\otimes \id\right)\circ\ldots\circ\left( \mu\otimes \id^{\otimes (n-2)}\right)\, , & n\ge 2
\end{array}
\right.
\end{eqnarray}
and where
\begin{equation}
(m^*)_{H^{\otimes n}}=\tau_{1n+1 \ldots n2n}\circ (m^*)^{\otimes n}\colon H^{\otimes n}\to  H^{\otimes n}\otimes  H^{\otimes n}
\end{equation}
is the canonical extension of the comultiplication
$m^*$ to the $n-$the tensor power of $H$,
admits the following explicit form:
$$
\omega^R_n (c)(a_1, \ldots, a_n)=\sum c(a_1', \ldots, a_n')i^* (a_1''\ldots a_n'')\,,
$$
where $a_i\in H$ for $1\le i\le n$ and $m^*(a_i)=\sum a_i'\otimes a_i''$ (in Sweedler notations). This allows to simplify computations. Indeed,
$$
\omega^R_{n+1}(\delta_0 c) (a_1, \ldots, a_n) = \sum a_1' c(a_2'\ldots, a_{n+1}') i^*(a_1''a_2''\ldots a_{n+1}'')
$$
From the anti-morphism property of $i^*$, we immediately get
\begin{align*}
\omega^R_{n+1}(\delta_0 c) (a_1, \ldots, a_n)
&= \sum a_1' c(a_2'\ldots, a_{n+1}') i^*(a_2''\ldots a_{n+1}'')i^*(a_1'') \\
&=\sum a_1' \omega^R_{n} (c)(a_2\ldots, a_{n+1}) i^*(a_1'')=\delta^{new}_0\left( \omega^R_{n}(c)\right)(a_1, \ldots , a_{n+1})\,.
\end{align*}
On the other hand,
\begin{align*}
\omega^R_{n+1}(\delta_{n+1} c) (a_1, \ldots, a_n)
&= \sum  c(a_2'\ldots, a_{n+1}') a_{n+1}' i^*(a_1''a_2''\ldots a_{n+1}'') \\
&=\sum  c(a_2'\ldots, a_{n+1}') a_{n+1}' i^*( a_{n+1}')  i^*(a_1''\ldots a_{n}'')=\delta^{new}_{n+1}\left( \omega^R_{n}(c)\right)(a_1, \ldots , a_{n+1})
\end{align*}
since $\sum a' i^{*}(a'')=e^* (a)1$ for any $a\in H$. The proof of the identity $\omega^R_{n+1}\circ \delta_i = \delta^{new}_i \omega^R_{n}$
for $i=1, \ldots, n$ is equally easy.
\end{proof}

%\vskip 2mm
\noindent The proof of Proposition \ref{prop:MC-1-isomorphism} will follow from Proposition \ref{prop:MC} and
 Lemma \ref{lem:MC_interchange} by assuming that $V=H$
together with the standard left- and right- comodule structure on it.
\end{proof}

\newpage

%%%%%%%%%%%%%%%%%%%%%%%%%%%%%%%%%%%%%%%%%%%%%%%%%%%%%%%%%%%%%
%%%%%%%%%%%%%%%%%%%%%%%%%%%%%%%%%%%%%%%%%%%%%%%%%%%%%%%%%%%%%
%%%%%%%%%%%%%%%%%%%%%%%%%%%%%%%%%%%%%%%%%%%%%%%%%%%%%%%%%%%%%
\section{Differential graded Lie groups} \label{DGLG}
In this short section we introduce the second ingredient of the differential graded Lie groups/algebras, namely the differential.

\subsection{Differential graded manifolds}

Recall that the starting point to define gradings in section \ref{GLG} was the commutative monoid $\Gamma$ with a particular element that we were calling $0$.
We suppose that it has an element that, together with its opposite if it exists, generates $\Gamma$,
we call it 1.
In the cancellative case, the only possibilities (up to isomorphism) are $(\ZZ,1)$, $(\NN,1)$ and $(\ZZ/n\ZZ,1)$

\begin{defi}
A \define{$Q$-structure} or equivalently a \define{homological vector field} on a graded manifold is a derivation of its structure sheaf of degree 1 which squares to zero.
A \define{differential graded (dg) manifold}   (equivalently, \define{$Q$-manifold}) is a graded manifold with a homological vector field.
\end{defi}

A \define{morphism of dg manifolds} is a morphism of graded manifolds which relates the homological vector fields in the following sense: given $f \colon (\calM_1,Q_1) \to (\calM_2,Q_2)$, recall that $f^\sharp \colon \tilde f^*(\gradedF(\calM_2)) \to  \gradedF(\calM_1)$.
We require that $f^\sharp \circ \tilde f^* \circ Q_2 =  Q_1 \circ f^\sharp \circ \tilde f^*$.

In this paper, the focus is mainly on $\N$-graded $Q$-manifolds and their morphisms (see also \cite{NP6}).

The \define{product of dg manifolds} as a graded manifold has a natural homological vector field.
One just checks that if $Q_1, Q_2$ are homological, so is $Q_1 \otimes \id + \id \otimes Q_2$.

Therefore we see that the above condition for multiplicativity of a vector field on a graded Lie group (\ref{eq:multvf2}) means exactly that multiplicaton $m^*$ is a dg morphism.

These definitions and observations combine into:
\begin{prop}
The category of dg manifolds is cartesian monoidal.

\end{prop}

\subsection{Differential graded Lie groups}

\begin{defi}
The category of \define{differential graded (dg) Lie groups} is the category of monoidal objects in the category of dg manifolds which are groups.
\end{defi}

Morphisms of dg Lie groups are defined in the natural way, and we thus obtain a \define{category of dg Lie groups} $\cat{dgLieGrp}$.

The body of a dg Lie group is a Lie group, and we have a ``body'' functor\\
$|{\;\cdot\;}| \colon \cat{dgLieGrp} \to \cat{LieGrp}$.

\subsubsection*{Example: The shifted tangent dg Lie group of a Lie group}

Let $M$ be a manifold.
We define the \define{shifted tangent bundle} $T[1]M$ as the algebra-ed space with underlying space $M$ and structure sheaf defined by
\begin{equation}
\calO_{T[1]M} (U) = \Omega^\bullet(U)
\end{equation}
the vector bundle of differential forms, for $U \subseteq M$ open, with the natural $\NN$-grading, and obvious restriction maps.
%Note: we could define $T[\gamma]M$, or say that we obtain it from $T[1]M$ by the morphism $\NN \to \Gamma$ which sends 1 to $\gamma$.

This is an $\NN$-graded manifold: if $U \subseteq M$ is the domain of a chart $\phi \colon U \to V$, then
\begin{equation}
\calO_{T[1]M} (U) \simeq \smoothF(\phi(U)) \otimes S V[1]^*.
\end{equation}
Its body is obviously $|T[1]M| = M$ itself.

This is a dg-manifold with homological vector field
$Q = d_{DR}
$, given by the De Rham differential.
More precisely, if $f \in \calO_{T[1]M}$, then locally one can consider $f \in \Omega^\bullet(U)$, and $Q_{DR}f$ then corresponds to $df \in \Omega^{\bullet+1}(U)$ (this is a legitimate definition since vector fields are local operators).

If $M$ is a Lie group $G_0$ with multiplication $m$, then $G = T[1]G_0$ is a dg Lie group with multiplication $T[1]m$ which we now define.
This will define the functor $T[1]$ from Lie groups to dg Lie groups.

The unit $e \colon \{\ast\} \to G$ is ``the same'' as that of $G_0$, that is, it is the composition $e \colon \{\ast\} \to G_0 \hookrightarrow G$, by which we mean that $e \colon \calC(G) \to \calC^\infty(\{\ast\}) = \RR$ is the evaluation at the unit $e \in G_0$ of the degree 0 component of a function on $G$.
This unit is  a dg morphism (of degree 0): it is graded, and preserves the homological vector fields.
Indeed, the homological vector field on $\{\ast\}$ is 0, so the condition reads
$(\calC(G) \to^Q \calC(\calN) \to^{e^*} \RR) = (\calC(G) \to^{e^*} \RR \to^0 \RR)$.
The right-hand side is obviously the zero map, so this means that the evaluation at $e$ of the degree 0 component of any function $Q(f)$ has to be zero.
This is obviously true since $Q$ is of degree 1 and $G$ is nonnegatively graded.

As for multiplication, if $m \colon G_0 \times G_0 \to G_0$ is the multiplication of $G_0$, then $T[1]m$ is naturally defined as follows:
If $f \in \calC(G)_0$, then $(mf)(x,y) = f(xy)$ for $x, y \in G_0$.
If $f = f_i(.) e^i$ where  $(e^i)$ is a basis of $\mathfrakg$ and the Einstein summation convention over repeating indeces is assumed,
then $((T[1]m)f)(x,y) =  f_i(xy) \big( (T_eL_x)_j^i e_2^j +  (T_eR_y)_j^i e_1^j \big)$.

By a straightforward computation (for degree 0 and 1) one shows that
 $(Q,Q) \circ m^* = m^* \circ Q$.

Summarizing, we obtain the following:
\begin{prop}
The dg-manifold $G = T[1]G_0$ with the above unit and multiplication is a dg Lie group.
\end{prop}

\subsubsection*{Example: The Chevalley--Eilenberg dg Lie group of a Lie algebra or a {\small DGLA}}

\textbf{Case of a Lie algebra.}
If $\mathfrakg$ is a Lie algebra, its Chevalley--Eilenberg cochain complex, $\bigwedge \mathfrakg^*$, can be viewed as the algebra (with the wedge product) of functions on the $\NN$-graded manifold $\CE(\mathfrakg)$.
Namely,
\begin{equation}
\calC(\CE(\mathfrakg)) = S \mathfrakg[1]^* = S \mathfrakg^*[-1].
\end{equation}
In particular, its body is a point.
This $\NN$-graded manifold can be made into a dg manifold with homological vector field
$Q = d_{CE}$, called the Chevalley--Eilenberg differential.
The usual Lie algebra bracket is then recovered as the $Q$-derived bracket of degree $-1$ vector fields -- the simplest example of
the derived bracket construction \cite{YKSbig}; and $Q^2 = 0$ corresponds precisely to the Jacobi identity of $\mathfrakg$.

This dg manifold $\CE(\mathfrakg)$ can be made into a commutative dg Lie group, defining the multiplication as the coproduct.
Namely, define
\begin{align*}
m^* \colon \bigwedge \mathfrakg^* &\longrightarrow  \bigwedge \mathfrakg^* \otimes \bigwedge \mathfrakg^*\\
f &\mapsto f \otimes 1 + 1 \otimes f
\end{align*}
on generators $f \in \mathfrakg^*$, which is enough by imposing that $m^*$ be an  algebra morphism which is unital, so $m^*(1) = 1 \otimes 1$
Define the unit $e^* \colon \bigwedge \mathfrakg^* \to \RR$ as the projection to the degree $0$ component, which is a unital algebra morphism.

The right-unit law reads
$( \id \otimes e^* ) \circ m^* = \proj_1^* \colon \bigwedge \mathfrakg^* \to \bigwedge \mathfrakg^*  \otimes \RR$, that is,
for $f \in \mathfrakg^*$,
$( \id \otimes e^* ) (f \otimes 1 + 1 \otimes f)
= f \otimes 1 + 1 \otimes 0
= f \otimes 1
= \proj_1^*(f)$, and similarly for the left-unit law.
Checking associativity is similar, and exactly the same as for the usual coproduct.
Moreover, the multiplication is obviously commutative, in the sense that $\tau \circ m^* = m^*$ where $\tau$ is the flip.

The inverse is given on generators by $inv \colon f \mapsto f_0 - f$, that is, $inv = i \circ e - \id$ where $i \colon \RR \to \bigwedge \mathfrakg^*$ is uniquely defined.
%or is it given by $(-1)^n \id$ in degree $n$?
This is the only dg Lie group structure here (cf. Cartier--Milnor--Moore theorem).

As for the multiplicativity of $Q$, recall that it induces on $\CE(\mathfrakg) \times \CE(\mathfrakg)$ the homological vector field $Q \otimes \id + \id \otimes Q$.
Then we have to check that $(Q \otimes \id + \id \otimes Q) \circ m^* = m^* \circ Q \colon \bigwedge \mathfrakg^* \to \bigwedge \mathfrakg^* \otimes \bigwedge \mathfrakg^*$.
If $f \in \mathfrakg$, then
$(Q \otimes \id + \id \otimes Q) (f \otimes 1 + 1 \otimes f) = Qf \otimes 1 + 1 \otimes Qf = m^*(Qf)$. \\
To summarize, we have proved:

\begin{prop}
The graded manifold $\CE(\mathfrakg)$ with the homological vector field $Q = d_{CE}$ and unit and multiplication as above is a commutative dg Lie group, called the \define{Chevalley--Eilenberg dg Lie group} of the Lie algebra $\mathfrakg$.
\end{prop}

\noindent \textbf{Graded case.}
We want to extend this construction from Lie algebras to {\small DGLA}'s.
Let $\mathfrakg$ be a non-positively graded {\small DGLA},
recalling the remark \ref{rmk:completion} about the algebra completion issues, we need this condition.
We define $\CE(\mathfrakg)$ in the same way as an $\NN$-graded manifold.
The only change which occurs is that the structure constants of $\mathfrakg$ take gradings into account as well: all the usual
equations (antisymmetry, Jacoby identity) include some signs, but the form remains very similar.

Recall that $\calC(\CE(\mathfrakg)) = S \mathfrakg[1]^* = S \mathfrakg^*[-1]$, take into account the shifts in gradings and consider
the homological vector field $Q = d_{CE} + d_\mathfrakg$.

Repeating almost verbatim the beginning of this subsection, one obtains the following
\begin{prop}
The graded manifold $\CE(\mathfrakg)$ with the homological vector field $Q = d_{CE} + d_\mathfrakg$ admits the structure of a dg Lie group, called the \define{Chevalley--Eilenberg dg Lie group} of the {\small DGLA} $\mathfrakg$.
\end{prop}

\begin{rmk}
The construction above obviously reminds of Lie algebroids, and inspires us to consider the question of integration of those, which we plan to address in future works.
\end{rmk}

\bigskip

\newpage

%%%%%%%%%%%%%%%%%%%%%%%%%%%%%%%%%%%%%%%%%%%%%%%%%%%%%%%%%%%%%
%%%%%%%%%%%%%%%%%%%%%%%%%%%%%%%%%%%%%%%%%%%%%%%%%%%%%%%%%%%%%
%%%%%%%%%%%%%%%%%%%%%%%%%%%%%%%%%%%%%%%%%%%%%%%%%%%%%%%%%%%%%
\section{Graded Harish-Chandra pairs and integration of {\small DGLA}'s} \label{HC}

The goal of this section is to show the relation between differential graded Lie groups and algebras.
First we explain how  {\small DGLA}s are obtained from {\small DGLG}s.
Then we present the result on the equivalence of categories of graded Lie groups and graded Harish-Chandra pairs ({\small GHCP}).
And as a final step we introduce the notion of {\small DGHCP} -- differential graded Harish-Chandra pairs thus concluding the
{\small DGLA} to  {\small DGLG} integration procedure.

\subsection{{\small DGLA}s of {\small DGLG}s}

\subsubsection*{The 1-cocycle associated to a multiplicative vector field}

If $Q \in \vf(\calG)$, define
\begin{equation}
\xi = Q \conv i^* = \mu \circ (Q \otimes i^*) \circ m^*.
\end{equation}

The identity $e^* \circ \mu = e^* \otimes e^*$, gives
$\xi_e
= e^* \circ \xi
= e^* \circ \mu \circ (Q \otimes i^*) \circ m^*
= (e^* \otimes e^*) \circ (Q \otimes i^*) \circ m^*
= (Q_e \otimes (e^* \circ i^*)) \circ m^*
= (Q_e \otimes e^* ) \circ m^*
= Q_e$, that is,
\begin{equation}
\xi_e = Q_e.
\end{equation}

%In the non-graded case, $\xi$ has values in a fixed tangent space, $\mathfrakg$ (actually, that is why we use $\xi$).

Using the results of section \ref{sec:mc} on the Maurer--Cartan endomorphism one proves that $\xi$ is a 1-cocycle.

\subsubsection*{The derivation associated to a multiplicative vector field}

If $X \in \vf(G)$, we define $\delta_X \colon \mathfrakg \to \mathfrakg$ by
\begin{equation}
\delta_X v = v \circ X.
\end{equation}

This notion is important in the following context:
\begin{prop}
If $X \in \multvf(G)$ has degree $d$, then $\delta_X \in \Der^d(\mathfrakg)$ is a derivation of degree $d$.
\end{prop}

\begin{proof}
The only thing to check is  the behaviour of $\delta_X$ with respect to the bracket
on $\mathfrak{g}$.
The result is: $\delta_X [v, w] = [\delta_X v, w] + (-1)^{d|v|} [v, \delta_X w]$.
It is obtained by computing $\delta_X [v, w]$ from its definition,
and using the multiplicativity of $X$ (\ref{defi:multvf1}). The sign appears due to the
grading since $deg X = d$, and it is precisely the same as for the degree $d$
derivation.
\end{proof}

Now it is easy to piece together the results from the previous parts and apply it to $Q$-structures. The homological condition  $Q^2 = 0$ immediately implies $\delta_Q^2 = 0$ since $\delta_Q v = v \circ Q$.

Among examples, let us mention the following two natural constructions:

\textbf{The {\small DGLA} of a shifted tangent dg Lie group}
$T[1]G$ is $\mathfrakg[1] \oplus \mathfrakg$, that is,
\begin{equation}
\dots \longrightarrow 0 \longrightarrow \mathfrakg[1] \longrightarrow \mathfrakg \longrightarrow 0 \longrightarrow \dots
\end{equation}
The bracket is constructed from the original bracket on $\mathfrakg$, and it does not
make a difference if it is computed on elements of $\mathfrakg$ or $\mathfrakg[1]$,
except for the case of  $[\mathfrakg[1], \mathfrakg[1]]$ which vanishes for degree reasons.
And the differential is $\id \colon \mathfrakg[1] \to \mathfrakg$.

\textbf{The {\small DGLA} of a Chevalley--Eilenberg dg Lie group.}
To start with, in the non-graded case the following proposition holds.
\begin{prop}
If $\mathfrakg$ is a Lie algebra, then the {\small DGLA} of $\CE(\mathfrakg)$ is the abelian {\small DGLA} $\mathfrakg[-1]$.
\end{prop}
Indeed, since the underlying manifold of $\CE(\mathfrakg)$ is a point, the degree 0 component of its {\small DGLA} is 0.
Also, since $\CE(\mathfrakg)$ is commutative, so should its DGLA be (that is, $[\cdot , \cdot]=0$, but in general its differential need not be zero).

Analogously, for {\small DGLA}s one has the following:
\begin{prop} \label{prop:neg}
If $\mathfrakg$ is a {\small DGLA}, then the {\small DGLA} of $\CE(\mathfrakg)$ is the abelian {\small DGLA}
$\mathfrakg[1]^\ast$ with the differential being the transpose of the original one (and reversed grading), that is:
\begin{equation*}
\xymatrix{
\mathfrakg \colon \qquad
\ldots \ar[r] & \mathfrakg_{-2} \ar[r]^d & \mathfrakg_{-1} \ar[r]^d & \mathfrakg_{0} \ar[r] & 0 \ar[r] & 0 \ar[r] & 0 \ar[r] & \ldots\\
\mathrm{DGLA}(\CE(\mathfrakg)) \colon
\ldots \ar[r] & 0 \ar[r] & 0 \ar[r] & 0 \ar[r] & \mathfrakg_{0}^* \ar[r]^{d^*} & \mathfrakg_{-1}^* \ar[r]^{d^*}  & \mathfrakg_{-2}^* \ar[r] & \ldots
}
\end{equation*}
\end{prop}

\subsection{Graded Harish-Chandra pairs}

In this subsection we define the graded Harish-Chandra pairs, by ``graded'' in this
and next section we mean $\NN$-graded (in contrast to $\ZZ$).
The construction
mimics essentially the super  case ($\ZZ_2$-graded), we thus follow
the summary in \cite{vish} of \cite{kostant} and \cite{koszul}. In this presentation we
will point out one essential difference: the $\ZZ_2$-graded case uses finite dimensionality of the graded part, which does not hold anymore in the $\NN$-graded case: elements of even degrees are not nilpotent, hence the formal power series do not reduce to polynomials.
Nevertheless, for a graded Lie algebra
one can construct directly a group law on the integrating object,
and when the {\small GLA} is differential
 with the construction of section \ref{sec:Qgroups} it  becomes a {\small DGLG}.

\begin{defi}
The \define{graded Harish-Chandra pair} is the following data:
\begin{itemize}
\item[-]  A couple $(G_0, \mathfrakg)$ of a Lie group
$G_0$ and a graded Lie algebra $\mathfrakg = \sum_{i \ge 0}\mathfrakg_i$,
for which $\mathfrakg_0 = Lie(G_0)$ is the Lie algebra of $G_0$
\item[-] A  degree preserving representation
 $(G_0, \mathfrakg)$ of a Lie group
$\alpha_{G_0}$ of $G_0$ in $\mathfrakg$ which induces the adjoint representation of
$G_0$ in $\mathfrakg_0$; and the differential  $(d \alpha_{G_0})_e$ of which at the identity $e \in G_0$ coincides with the adjoint representation $ad$ of $\mathfrakg_0 \in \mathfrakg$.
\end{itemize}
\end{defi}
\begin{rmk}
 In the definition above by ``graded'' we mean $\NN$-graded, and we write it as
 if $\NN = \ZZ_{\ge 0}$, i.e. non negatively graded.
 But it is important to note that there is no reason to disregard the non-positively graded
 case ($\NN = \ZZ_{\le 0}$), especially since it appears naturally passing to the dual of the picture (see for instance, Prop. \ref{prop:neg}). We will stress this fact in the final theorem.

\end{rmk}

The morphisms of graded Harish-Chandra pairs are defined in a natural way.
For  two Harish-Chandra pairs a \define{morphism}
 ${\cal F} \colon (G_0, \mathfrakg) \to (H_0, \mathfrakh)$  consists of a pair of
homomorphisms $f \colon G_0 \to H_0$ and $\mathfrak{f} \colon \mathfrakg \to  \mathfrakh$,
such that $(d f)_e = \mathfrak{f}|_{\mathfrak{g}_0}$ and
$\mathfrak{f} \circ \alpha_{G_0} (g) = (\alpha_{H_0}(g))\circ\mathfrak{f}, \; \forall g\in G$.

This defines the \define{category of graded Harish-Chandra pairs} that we denote
$\cat{GHCP}$. We will show that it is equivalent to the category of graded Lie groups.
One way of this equivalence is rather straightforward.
Given a graded Lie group $G$, on considers its body part $|G| = G_0$
together with the graded Lie algebra $\mathfrakg = Lie(G)$ and equips it with the
the adjoint representation $\alpha_{G_0} = Ad_{G_0}$.
The construction the other way around is a bit technical, we will sketch the essential
points of it here.

Let $\mathfrakU$ denote the universal enveloping graded-algebra functor.
If $\mathfrakg$ is a graded Lie algebra over $k$, then $\mathfrakU(\mathfrakg)$,
is a $\mathfrakU(\mathfrakg_0)$ module, and the action of $\mathfrakg_0$
on the sheaf $\gradedF_G(U)$ induces a structure of $\mathfrakU(\mathfrakg_0)$-module on $\gradedF_G(U)$.
From the graded Harish-Chandra pair, define then the graded manifold
structure sheaf as
\begin{equation}
\calO_G (U) = \Hom_{\mathfrakU(\mathfrakg_0)} \big( \mathfrakU(\mathfrakg), \calC^\infty(U) \big)
\end{equation}
for open subsets $U \subseteq G_0$.
By the graded Poincar\'e--Birkhoff--Witt (PBW, \cite{felix}) theorem we have
\begin{equation}
\Hom_{\mathfrakU(\mathfrakg_0)} \big( \mathfrakU(\mathfrakg), \calC^\infty(U) \big)
\simeq
\Hom_k \left( S (\mathfrakg/\mathfrakg_0), \calC^\infty(U) \right).
\end{equation}
The graded enveloping algebra $\mathfrakU(\mathfrakg)$
can be equipped with a graded Hopf algebra structure, we can thus profit from all the
constructions from section \ref{GLG}.

The explicit construction of the above structure, as well as the description of the relation
of objects and morphisms of the mentioned categories goes through
verbatim as in \cite[Section 2.]{vish}, replacing the word ``super'' by ``graded''. We repeat  here the smooth version of this technique
(the generalization to the analytic and algebraic case is straightforward).

%\vskip 2mm
\noindent The graded Hopf algebra obtained from a Harish-Chandra pair is now
\begin{align}\label{eqn:smooth_HC}
H=
& \Hom_{\mathfrakU(\mathfrakg_0)} \big( \mathfrakU(\mathfrakg), \calC^\infty(G_0) \big) =\\ \nonumber
& \{f\in \Hom \big( \mathfrakU(\mathfrakg), \calC^\infty(G_0) \big)\, |\, f(uX, g)= -\left(X^R f\right)(u,g)\,, \forall \, u\in  \mathfrakU(\mathfrakg),
X\in \mathfrakg_0 , g\in G_0\}\,,
\end{align}
where $X^R$ is the right-invariant vector field on $G_0$ corresponding to $X$. The (graded commutative) multiplication is the convolution product, i.e.
it is defined as
$$
(f_1 f_2)(u,g)=(f_1\otimes f_2)(\Delta u, g,g),
$$ where $\Delta$ is the standard comultiplication in $\mathfrakU(\mathfrakg)$, while the (graded)  comultiplication
$m^*\colon H\to H\otimes H$
is the co-convolution product, i.e. $m^*(f)(u_1, g_1, u_2, g_2)=f(u_1\alpha_{G_0}(g_1,u_2), g_1g_2 )$. It is not hard to verify that the product and coproduct
of elements of $H$ belongs to $H$ and $H\otimes H$, respectively. The antipode is obtained as a combination of the antipodes in $\mathfrakU$ and $C^\infty (G_0)$.

To sum it up, the following theorem holds:\\[-2em]
\begin{thm}
\label{thm:main}
There is an equivalence of categories between non-negatively graded Lie groups and non-negatively graded Harish-Chandra pairs.
\end{thm}

There are however two very important points to mention.
First, even if the construction is very similar to the super case, the essential difference is in the definitions of the employed  structures and in particular the graded Hopf
algebras (section \ref{GLG}).
Second, the construction relies heavily on the PBW theorem, and there it is important
that the grading is $\NN$ (i.e. $\ZZ_{\le 0}$ or $\ZZ_{\ge 0}$ but not $\ZZ$), meaning that there is no problem in consistent
ordering of the basis of $(\mathfrakg/\mathfrakg_0)$. The construction
may be applied in some more general cases, but then a lot of technicalities occur.
We are going to discuss the question of validity of PBW in a separate paper \cite{KPS}.

\subsection{Integration of {\small DGLA}s}
\label{sec:Qgroups}

The idea of the method is the following:
Given an $\NN$-graded {\small DGLA} $\mathfrakg$, one integrates its degree 0 part $\mathfrakg_0$ to its simply connected Lie group $G_0$.
This gives a graded Harish-Chandra pair $(G_0, \mathfrakg, \alpha_{G_0})$.
One constructs its associated graded Lie group as in the previous subsection, and
finally, constructs the homological vector field from the differential on the {\small DGLA} --
we detail this step in the current section.

%\vskip 2mm
\noindent Let us extend $\alpha_{G_0}$ to all graded derivations $\mathrm{Der}^\bullet (\mathfrakg)$
of $\mathfrakg$ by use of the conjugation; given that $\alpha_{G_0}(g,-)$ is a degree $0$ automorphism of $\mathfrakg$ for every $g$,
the conjugation  of any graded derivation by $\alpha_{G_0}(g,-)$ is a derivation of the same degree. For any connected $G_0$ and any
$\delta\in \mathrm{Der}^\bullet (\mathfrakg)$ one has $\alpha_{G_0}\delta\alpha_{G_0}^{-1}=\delta $ modulo inner derivations.
Moreover, if we denote \\
$\bar\lambda (g) \colon = \alpha_{G_0}(g,-)\delta\alpha_{G_0}(g,-)^{-1}-\delta$, then $\bar\lambda$ is a 1-cocycle on $G_0$ with values in the space inner derivations of degree 1 regarded as a $G_0-$module by use of the conjugation by ${\alpha_{G_0}}$. Indeed, for any $g,h\in G_0$
we obtain
\begin{align*}
\bar\lambda (gh) = \alpha_{G_0}(gh,-)\delta\alpha_{G_0}(gh,-)^{-1}-\delta =
\alpha_{G_0}(g,-)\left( \alpha_{G_0}(h,-)\delta\alpha_{G_0}(h,-)^{-1}-\delta\right)\alpha_{G_0}(g,-)^{-1} +\\
+ \alpha_{G_0}(g,-)\delta\alpha_{G_0}(g,-)^{-1}-\delta = \alpha_{G_0}(g,-)\bar\lambda(h)\alpha_{G_0}(g,-)^{-1} + \bar\lambda (g)\,.
\end{align*}
This motivates the following definition.

\begin{defi}%[Differential graded Harish-Chandra pair]
\label{def:DGHC} Let  $(G_0, \mathfrakg, \alpha_{G_0})$ be a graded Harish-Chandra pair,
$\mathfrakg$ be a {\small DGLA} over a field $k$ with a differential $\partial$. We call $(G_0, \mathfrakg, \partial,\alpha_{G_0})$
a \define{differential graded Harish-Chandra pair} (DG Harish-Chandra pair) if there exists a $\mathfrakg^1-$valued $1-$cocycle on $G_0$, i.e. a smooth map
$\lambda\colon G_0\to \mathfrakg^1$ which satisfies
\begin{equation}\label{eqn:lambda-cocycle}
\lambda(gh)=\lambda(g) +\alpha_{G_0}(g,\lambda(h))
\end{equation}
for all $g,h\in G_0$, such that in addition
\begin{equation}\label{eqn:lambda}
\bar\lambda (g)= ad\circ\lambda (g)\,.
\end{equation}
\end{defi}

\begin{rmk}\label{rem:inner-outer}
If $\partial$ is an inner derivation then $\lambda$ is uniquely fixed by $\alpha_{G_0}$. Otherwise the identity (\ref{eqn:lambda}) will fix $\lambda$
only modulo the center of $\mathfrakg$.
\end{rmk}

\begin{rmk}
 Spelling out the definition of morphisms of DG Harish-Chandra pairs is an instructive exercise.
\end{rmk}

\begin{lem}\label{lem:extension_of_bar_lambda}
Let $(\mathfrak{g},G_0,\alpha_{G_0})$ be a Harish-Chandra pair with a simply connected base $G_0$  and $\partial$ be a degree one outer differential in $\mathfrakg$. Then there exists a canonical extension
of $\bar\lambda$ to $\lambda$, which makes $(G_0,\mathfrak{g}, \alpha_{G_0}, \lambda)$  a differential graded Harish-Chandra pair.
\end{lem}
\begin{proof}
 The differential of $\lambda$ at the identity must give us the following 1-cocycle on $\mathfrak{g}_0$: $T_e G_0\ni X\mapsto -\partial (X)\in \mathfrakg^1$; since $G_0$ is simply connected
this uniquely determines the required 1-cocycle $\lambda$ on $G_0$ by the Van Est isomorphism.
\end{proof}

\begin{thm}
\label{thm:main2}
There is an equivalence of categories between $\NN$-graded differential Lie groups and  differential $\NN$-graded Harish-Chandra pairs.
\end{thm}
\begin{proof}
Let $(\mathfrak{g},G_0,\alpha_{G_0}, \partial, \lambda)$ be a DG Harish-Chandra pair. If $\partial$ is an inner derivation corresponding to a degree $1$ element of $\mathfrakg^1$ which we denote by the same letter (by Remark \ref{rem:inner-outer}
$\lambda$ is uniquely fixed by $\alpha_{G_0}$), then we define a multiplicative structure as the difference between left- and right- translations of $\partial$.
By use of Prop. \ref{prop:exact} this is a multiplicative vector field; it is easy to see that this vector field will give us back the differential $\partial$
in $\mathfrakg$.

\noindent More precisely, the multiplicative vector field $Q=\partial^L-\partial^R$ acts on an arbitrary smooth function $f$ on $G$ as follows:
\begin{align*}
(Qf)(u,g) &=(\partial^L f)(u,g)-(\partial^R f)(u,g)=(-1)^{\mathrm{deg}(u)}f(u{\alpha}_{G_0}(g,\partial),g)-f(\partial u,g)
\end{align*}
for any $u\in \mathfrakU({\mathfrakg})$, $g\in G_0$. On the other hand,
\begin{equation}\label{eqn:inner_multiplicative}
(Qf)(u,g) = (-1)^{\mathrm{deg}(u)}f(u\lambda(g),g)- f([\partial, u],g)\,,
\end{equation}
where $\lambda(g)={\alpha}_{G_0}(g,\partial)-\partial$.
Indeed,
\begin{align*}
(Qf)(u,g) &= (-1)^{\mathrm{deg}(u)}f(u\tilde{\alpha}_{G_0}(g,\partial),g)-f(\partial u,g) \\
              &= (-1)^{\mathrm{deg}(u)}f(u{\alpha}_{G_0}(g,\partial),g) - f([\partial, u],g)-(-1)^{\mathrm{deg}(u)}f(u\partial ,g) \\
              &= (-1)^{\mathrm{deg}(u)}f(u({\alpha}_{G_0}(g,\partial)-\partial),g)- f([\partial, u],g) \\
              &= (-1)^{\mathrm{deg}(u)}f(u\lambda(g),g)- f([\partial, u],g)\,,
\end{align*}

%\vskip 2mm
\noindent Now we use formula (\ref{eqn:inner_multiplicative}) to extend the integration procedure to the more genaral case as follows.
Let $\partial$ be an outer derivation; we apply Lemma \ref{lem:extension_of_bar_lambda}
to obtain a 1-cocycle $\lambda$ and thus the structure of a DG Harish-Chandra pair (see Definition \ref{def:DGHC}). By replacing of $[\partial, u]$ with $\partial(u)$
in (\ref{eqn:inner_multiplicative}), we obtain the formula for the multiplicative vector field $G$ on $G$:
\begin{equation}\label{eqn:outer_multiplicative}
(Qf)(u,g) = (-1)^{\mathrm{deg}(u)}f(u\lambda(g),g)- f(\partial (u),g)
\end{equation}
for all $u\in \mathfrakU({\mathfrakg})$, $g\in G_0$. The rest of the proof including the morphism property is straightforward.
\end{proof}

\subsubsection*{Extended Harish-Chandra pairs}

\begin{lem}\label{lem:extended_HC}
 Let $\mathfrak{g}$ be a {\small DGLA} over a field $k$ with an outer differential $\partial$. Then
\begin{itemize}[itemsep=-0.3em, leftmargin=-0.3em]
\item
$\tilde{\mathfrak{g}}=\mathfrak{g}\oplus k\partial$  admits a canonical structure of a {\small DGLA}, such that $\mathfrak{g}$
is a graded Lie subalgebra, $\partial^2=0$ and $[\partial, X]=\partial (X)$ for every $X\in\mathfrak{g}$. The differential
in $\tilde{\mathfrak{g}}$ is given by the adjoint action of $\partial$;
\item a $1-$cocycle $\lambda$ from Def. \ref{def:DGHC} is in one-to-one correspondence with an extension $\tilde{\alpha}_{G_0}$ of
 $\alpha_{G_0}$ to $\tilde{\mathfrak{g}}$;
\item if $G_0$ is simply connected then there exists a canonical extension
$\tilde{\alpha}_{G_0}$ of
$\alpha_{G_0}$ to $\tilde{\mathfrakg}$, which makes $(\tilde{\mathfrak{g}},G_0,\tilde{\alpha}_{G_0})$ into a Harish-Chandra pair.
\end{itemize}
\end{lem}
\begin{proof} While the first two statements are resulting from a straightforward computation, the third one follows from the second statement combined with  Lemma \ref{lem:extension_of_bar_lambda}.
\end{proof}

\begin{defi}\label{def:extended_HC}
 We shall call $(\tilde{\mathfrak{g}},G_0,\tilde{\alpha}_{G_0})$ an \define{extended Harish-Chandra pair}.\footnote{The idea to interpret the integration of DGLA with an outer derivation in
terms of such an extended pair was suggested to us by C.~Laurent-Gengoux.}
\end{defi}

\noindent By construction, the extended Harish-Chandra pair $(\tilde{\mathfrak{g}},G_0,\tilde{\alpha}_{G_0})$ integrates the (extended) graded Lie algebra $\tilde{\mathfrakg}$
to a graded Lie group $\tilde{G}$ with a graded subgroup $G$, which corresponds to the initial Harish-Chandra pair $({\mathfrak{g}},G_0,{\alpha}_{G_0})$.
Taking into account that $\partial$ is now an inner derivation of $\tilde{\mathfrakg}$, we can integrate it to a multiplicative vector field $\tilde{Q}$ on $\tilde{G}$ by use of formula (\ref{eqn:inner_multiplicative}).

\begin{lem}\label{lem:extension_to_DGHC} $G$ is a differential graded Lie subgroup of $\tilde{G}$, such that the induced DGLG structure on $G$ coincides with the one given by
formula (\ref{eqn:outer_multiplicative}).
\end{lem}
\begin{proof}
  Notice that
 the ideal of $G$ in the graded algebra of smooth functions on $\tilde{G}$, i.e. the ideal of functions vanishing on $G$ is
\begin{align}
 I_G=\{f\in \Hom_{\mathfrakU(\mathfrakg_0)} \big( \mathfrakU(\tilde{\mathfrakg}), \calC^\infty(G_0) \big)\, | \,
f\left( \mathfrakU({\mathfrakg}) =0 \right) \,.
\end{align}
If $u\in \mathfrakU({\mathfrakg})\subset \mathfrakU(\tilde{\mathfrakg})$ then $u\lambda(g)$ and $[\partial, u]\equiv\partial (u)$
also belong to $\mathfrakU({\mathfrakg})$, therefore for any $f\in I_G$ one has $(\tilde{Q}f)(u,g)=0$ and thus $\tilde{Q}f\in I_G$. Finally the restriction of $\tilde{Q}$
onto $G$ defines the multiplicative
structure $Q$ on $G$ which gives back $\partial$ in $\mathfrakg$ and the formula for $Q$ coincides with (\ref{eqn:outer_multiplicative}).
\end{proof}

\subsubsection*{Examples and exercises}
This construction reverses the procedure described above of ``differentiating''
of a {\small DGLG} to a {\small DGLA}.
 It can for instance be applied for the  \textbf{examples} from the  previous section, namely recover:  the shifted tangent bundle to a Lie group; the Chevalley--Eilenberg Lie group in the graded case.
A motivated reader may also consider simpler examples (i.e. specifications) like:
  the dg Lie group of an abelian {\small DGLA}; the dg Lie group of a {\small DGLA} concentrated in degree $d$.

\bigskip

\newpage
\section{Conclusion / Discussions}
In this paper we addressed the question of integrating differential graded Lie algebras to differential graded Lie groups.
As mentioned in the introduction, this is a part of a big project of a systematic study of the integration problem on the categorical level:
it should include among others some $\infty$ structures and generalized geometry, with potentially non-trivial links between them.

Let us stress again, even if initially the strategy of this paper meant to repeat essentially the approach of \cite{vish} in the case of super
{\small DLG}s and {\small DLA}s (i.e. $\ZZ/2\ZZ$-graded) and add ``by hand'' a $Q$-structure to it, the question turned out to be more intricate: working with $\ZZ$- and even $\NN$- graded objects presents  conceptual challenges. So the resemblance of the final construction
for the $\NN$-graded case
to the super case is misleading: it relies on the results that are \emph{not}  straightforward generalizations, and therefore had to be explicitly explained.

Two points are worth  mentioning here: \\
First. The main result concerns equivalence of categories, and there graded Harish-Chandra pairs play the key role. The concept of differential graded Harish-Chandra pairs
that we introduced, is an important step -- those seem to have higher analogues and actually give a possible way to generalize the result
to Lie algebroids and possibly other structures. \\
Second. As we have understood from the section \ref{HC}, the construction works as long as one can safely apply the
Poincar\'e--Birkhoff--Witt theorem.
But the tricky point is before that, already at the level of definition of the functional spaces on graded algebras/groups.
Namely natural elements are now formal power series in graded variables, not polynomials -- one thus loses some intuition about their
behaviour (see appendix \ref{sec:funcan} to get some flavour).
We thought of it as an auxiliary technical issue, but again in the $\ZZ$-graded case it turned out to be more interesting.
We realized that
 careful description of the functional space, the universal envelopping algebra with its properties, as well as the Hopf algebra related
questions, is a problem worth being detailed by itself. Thus, not to overload the presentation here, we are going to devote a separate paper (\cite{KPS}) exclusively to this topic.

\textbf{Acknowledgments.}

The research of A.K. was supported by grant no. 18-00496S of the Czech Science Foundation.
V.S. started working on this project in the University of Luxembourg, his research at that time was supported by the Fonds National de la Recherche,  project F1R-MTH-AFR.\\
V.S. is also thankful to the La Rochelle University for the
Young Researcher's Grant (``Action Incitative Jeune Chercheur''), that permitted to arrange several meetings of him and A.K. at the final stages of this work.

\newpage
\appendix
\setlength{\parskip}{0mm}

\section{Lie group and Lie algebra cohomologies}\label{AppCoho}

\textbf{Lie group cohomology.}

Let $G$ be a Lie group and let $A$ be a smooth $G$-module, i.e. an Abelian Lie group endowed with a smooth $G$-action $\zr:G\times A\to A$. For $g\in G$, we write $\zr(g,-)=g\cdot -$, and, for $a,a'\in A$, we have $g\cdot(a+a')=g\cdot a+g\cdot a'$.

The cochain complex for the smooth cohomology of the Lie group $G$ `represented' on $A$ by $\zr$ is defined by
\begin{equation}
C_{\op{sm}}^n(G,\zr) = \smoothF(G^{\times n},A),\;n\in\N\;,
\end{equation}
where $C_{\op{sm}}^0(G,\zr) := A$.
The coboundary map $d$ for smooth group cohomology is the same as for ordinary group cohomology,
\begin{equation}
d\xi(g_0,\ldots, g_n) = g_0 \cdot \xi(g_1,\ldots, g_n) + \sum_{i=1}^n (-1)^i \xi( \ldots, g_{i-1}g_{i}, \ldots) + (-1)^{n+1} \xi(g_0,\ldots g_{n-1})\;.
\end{equation}

In particular, %$$ d\xi(g_0) = g_0 \cdot \xi - \xi\;,$$ so $H^0(G,\zr) = Z^0(G,\zr) = A^G$, with obvious notation, and
%\begin{equation}
$ d\xi(g_0,g_1) = g_0 \cdot \xi(g_1) - \xi(g_0g_1) + \xi(g_0). $
%\end{equation}

Hence, if \be\Ad:G\ni g\mapsto \Ad_g=T_{g^{-1}}L_g\circ T_eR_{g^{-1}}=T_gR_{g^{-1}} \circ T_eL_g\in \op{Aut}(\mathfrak{g})\ee is the adjoint representation of $G$ on its Lie algebra $\mathfrak g$, a 1-cocycle $\xi\in Z_{\op{sm}}^1(G,\Ad)$, is a map $\xi \in \smoothF(G,\mathfrakg)$ that satisfies the equation
\begin{equation}
\label{eq:cocycle}
\xi(gh) = \Ad_g (\xi(h))  + \xi(g)\;,
\end{equation}
for any $g, h \in G$.

\textbf{Lie algebra cohomology.}

Let $\mathfrak g$ be a Lie algebra (with bracket $[-,-]$) and let $\zr:\mathfrak{g}\to \op{End}(V)$ be a representation of $\mathfrak g$ on a vector space $V$.

The cochain complex for the Chevalley-Eilenberg cohomology of the Lie algebra $\mathfrak g$ represented on $V$ by $\rho$ is defined by
\begin{equation}
\CCE^n(\mathfrakg,\rho) = {\cal A}(\mathfrakg^{\times n}, V),\;n\in\N\;,\end{equation} where the {\small RHS} is the space of $n$-linear antisymmetric maps from $\mathfrak g$ to $V$ and where $\CCE^0(\mathfrakg,\rho) := V$.
The coboundary map $d$ is given by

\begin{equation}
d\omega(X_0, \ldots, X_n) = \sum_{i=0}^n (-1)^i \rho(X_i)\left( \zw(X_0,\dots \hat{\imath} \dots, X_n)\right)
\end{equation}
$$+\sum_{i<j} (-1)^{i+j} \zw([X_i,X_j],X_0, \dots \hat{\imath}\dots \hat{\jmath} \dots ,X_n)\;,$$ with standard notation.

In particular, %$$d\zw(X_0) = \zr(X_0)(\zw)\;,$$ so $H^0(\mathfrakg,\rho) = Z^0(\mathfrakg,\rho) = V^\mathfrakg$, and
for the adjoint representation $$\op{ad}:\mathfrak{g}\ni X\mapsto \op{ad}_X=[X,-]\in\op{Der}(\mathfrak{g})\;,$$ we have $$d\zw(X,Y) = [X,\zw(Y)] - [Y,\zw(X)] - \zw([X,Y])\;,$$ so $Z_{\op{CE}}^1(\mathfrakg, \ad) = \Der(\mathfrakg)$.

\newpage

\setlength{\parskip}{3mm}

\section{Graded manifolds} \label{sec:gr-man-app}

In this appendix we recall (or introduce) some definitions related to graded manifolds.
The approach is rather similar to that of~\cite{ccf}, which treats the $\ZZ/2\ZZ$-graded (``super'') case, the main point is to make some ``folkloric'' statements
explicit and fix the notations for the current paper to make it self-consistent.

\subsection{Graded manifolds -- definition}

Let $\Gamma$ be a commutative monoid and $\varepsilon \colon \Gamma^2 \to \RR^\times$ is a commutation factor (see~\cite[III.46]{bou}) and $V$ is a $\Gamma$-graded vector space, we define its \define{graded symmetric algebra}
\begin{equation}
\grsym V = \otimes V \big/ \left\langle v \otimes w - \varepsilon(v,w) \, w \otimes v \mid v, w \in V^{\text{hom}} \right\rangle
\end{equation}
where we write $\varepsilon(v,w)$ for $\varepsilon(|v|,|w|)$, and $|\bullet|$ denotes the degree of $\bullet$.
Since the ideal by which we quotient is homogeneous (for the $\Gamma$-grading),
 the graded symmetric algebra is a $\Gamma$-graded $\varepsilon$-commutative unital\footnote{By definition, a graded algebra is \define{unital} if it has a multiplicative unit which is homogeneous of degree 0.
} algebra.
The ``super'' case corresponds to $\Gamma = \ZZ_2 \equiv \ZZ/2\ZZ$ and $\varepsilon(\gamma_1,\gamma_2) = (-1)^{|\gamma_1| |\gamma_2|}$.

To define the degree of graded linear maps,
we need $\Gamma$ to be cancellative, which is equivalent to being embeddable in a commutative group.
In the following, ``graded'' will mean ``$\Gamma$-graded'' for some fixed commutative monoid $\Gamma$, and ``commutative'' will mean ``$\varepsilon$-commutative'' for an $\varepsilon$ usually left implicit.
In most of this article, $\Gamma = \ZZ/2\ZZ$, $\NN$ (meaning $\ZZ_{\ge 0}$), or $\ZZ$, that is the ``degree zero'' actually
corresponds to the element $0 \in \Gamma$,
and the commutation factor will be the one given above.

If $U$ is an open subset of  $\RR^n$ and $V$ is a graded vector space, we define the unital graded $\RR$-algebra
\begin{equation}
\gradedF(U|V) = \smoothF(U) \otimes \grsym V
\end{equation}
and we call it an \define{algebraic model}.
It is $\varepsilon$-commutative.
Quotienting it by the ideal generated by the homogeneous elements of nonzero degree, we obtain the unital graded $\RR$-algebra isomorphism $\gradedF(U|V) \big/ \gradedF(U|V)^{\neq 0} \simeq \gradedF(U|V_0)$.
If $V_0 = \{0\}$, the quotient map
$\gradedF(U|V) \twoheadrightarrow \gradedF(U|V) \big/ \gradedF(U|V)^{\neq 0} \simeq \smoothF(U)$ is denoted by $f \mapsto \widetilde{f} = f_\varnothing$.

By abuse of notation, we also denote by $\gradedF(U|V)$ the unital-graded-algebra-ed space $\left( U, \gradedF(-|V) \right)$ it naturally defines, and we call it a \define{local model}.
A \define{``something''-ed\footnote{The etymology comes from ``ring'' -- ``ringed'' often appearing in the literature.}
 space} is a topological space with a sheaf of ``something''s on it called its \define{structure sheaf}.

A \define{morphism} of these is a pair $\phi = (\widetilde{\phi}, \phi^\sharp)$ where $\widetilde{\phi}$ is a continuous map between the underlying spaces and $\phi^\sharp$ is a sheaf morphism from the pullback by
$\widetilde{\phi}$  of the target sheaf to the  source sheaf:
$$
     \phi \colon \gradedF(U|V) \to \gradedF(U'|V') \quad \Leftrightarrow \quad
     \tilde  \phi \colon U \to U' \text{ and }
     \phi^\sharp \colon \tilde \phi^*({\gradedF(U'|V')}) \to \gradedF(U|V)
$$
Providing the data of $\phi = (\tilde \phi, \phi^\sharp)$ is equivalent to
the ``dual'' construction
$\psi = (\tilde \phi, \phi_\sharp)$, where
$\phi_\sharp \colon {\gradedF(U'|V')} \to \tilde \phi_*{\gradedF(U|V)}$ (\cite{shafar}).

For brevity, we will write ``algebra-ed'' for ``unital-graded-algebra-ed''.

\begin{lem}
\label{lem:local}
If $\gradedF(U|V) \simeq \gradedF(U'|V')$ either as algebra-ed spaces, or as graded unital algebras with $V_0=V'_0=\{0\}$,
then $U \simeq U'$ and $V \simeq V'$.
\end{lem}

\begin{proof}
If they are isomorphic as spaces, then by definition $U \simeq U'$, else by the above $\gradedF(U|V_0) \simeq \gradedF(U'|V'_0)$ as quotients, and from the hypothesis, $\smoothF(U) \simeq \smoothF(U')$ and it follows by a classical result that $U \simeq U'$.
Now, if one considers the subalgebra generated by elements of a given degree $\gamma$, which is an isomorphism invariant, one obtains $\smoothF(U) \otimes \grsym V_\gamma \simeq \smoothF(U) \otimes \grsym V'_\gamma$. The ranks of the modules of derivations of these algebras are respectively $\dim U + \dim V_\gamma$ and $\dim U + \dim V'_\gamma$, so $V_\gamma$ and $V'_\gamma$ have the same dimension, so are isomorphic, and $V \simeq V'$.
\end{proof}

\begin{defi}
A \define{graded manifold} is a paracompact Hausdorff unital-graded-algebra-ed space, locally modelled as
$\gradedF(U|V)$, where $U$ is an open subset of an $\RR^n$ and $V$ is a graded vector space with $V_0 = \{0\}$.
\end{defi}

\begin{rmk}
The Lemma \ref{lem:local} guarantees that a graded manifold is well defined, and sometimes in literature it is not
proven but included in the definition
of a ``graded manifold of body-dimension $n$ and modelled on $V$''.
\end{rmk}

A \define{morphism} of graded manifolds is a morphism of algebra-ed spaces.
In other words, the category $\cat{GMan}$ of graded manifolds is a full subcategory of the category of algebra-ed spaces.
In particular, morphisms are of degree 0.

We denote by $\calO_\calM$ the structure sheaf of the graded manifold $\calM$.
If  $\calM$ is a graded manifold, the topological space $|\calM|$, which is covered by
open sets $U$ from couples $(U|V)$,
inherits the structure of a (non-graded) manifold since we saw that $\smoothF(U)$ can be recovered naturally from $\gradedF(U|V)$.
It is called the \define{body} of $\calM$ and is sometimes also denoted by $\calM_0$, $\widetilde{\calM}$ or $\bar{\calM}$.
This gives a functor $\cat{GMan} \to \cat{Man}$ which is a retraction (hence full and surjective).
Any (smooth) manifold is considered as trivially graded --- that is the functor is  not faithful,  what makes graded manifolds interesting.

\begin{rmk} \label{rmk:completion}
Defined like this, the notion of a graded manifold is enough for the purpose of this paper, namely for the $\NN$-graded case. For the general $\ZZ$-graded
situation one may need a suitable completion of the algebra of graded polynomials discussed above. For instance, in the category of $\Z_2^n$-graded manifolds, which was introduced and studied, together with the corresponding $\Z_2^n$-Berezinian and (low-dimensional) $\Z_2^n$-integration-theory, in \cite{NP4, NP5, NP1}, formal power series are unavoidable. In fact most of the constructions that will follow in this section and section \ref{DGLG} remain valid in that case as well. The subtleties occur for the construction of the graded Harish-Chandra pairs (cf. section \ref{HC}) --
we are going to address this question in a separate paper (\cite{KPS}).

\end{rmk}

The category of graded manifolds is a (full) subcategory of the category of locally algebra-ed spaces.\footnote{In the real case we are considering the construction is similar to locally ringed spaces. In the complex case apparently there may be subtleties, but they appear already for the base manifold, so this is not an issue specific to grading. The way out in the complex case is to consistently use the sheaf-theoretic terminology, like in \cite{vish}.}
A consequence of locality is the following.
\begin{prop}
If $\phi = (\widetilde{\phi},\phi^\sharp) \colon \calM \to \calN$ is a morphism of graded manifolds, then for all open subsets $U \subseteq |\calN|$ and functions $f \in \calO_\calN(U)$, one has
$
\widetilde{\phi^\sharp(f)}
= \widetilde{\phi}^* \left( \widetilde{f} \right)
$
viewed
as functions in $\smoothF \left( \widetilde{\phi}^{-1}(U) \right)$.
\end{prop}

\textbf{Functor to supermanifolds.}
If the commutation factor is  trivial, then a graded manifold (with $V$ finite-dimensional) can be seen as a usual manifold up to completion.
Indeed, we first forget the grading, and then we complete the algebra of functions, using the fact that polynomials are dense in the usual topology.
Graded morphisms are then mapped to smooth morphisms (this is possible since the sheaf component of a graded morphism can be defined by specifying only its restriction to $\smoothF(U)$ and the finite dimensional space $V$).
For instance, $\gradedF(U|\RR^n[\gamma]) \mapsto \smoothF(U) \otimes \grsym(\RR^n[0]) \mapsto \smoothF(U \times \RR^n)$.
This functor is different from the ``body'' functor, which is a subfunctor of that one.

If $\Gamma = \NN$ or $\ZZ$ and the commutation factor is the standard nontrivial one, then there is a functor from the category of  graded manifolds (with $V$ finite-dimensional) to the category of supermanifolds.
It is obtained by the map $\NN$ or $\ZZ \to \ZZ/2\ZZ$, and then completing the algebra of functions of degree 0, for instance, $\gradedF(U | \RR^n[2k]) \mapsto \smoothF(U) \compTens \grsym(\RR^n[0]) \mapsto \smoothF(U \times \RR^n)$.

\subsection{Products of graded manifolds}

We now turn to the question of products of graded manifolds.
The binary coproduct of the algebraic models in the category of unital graded commutative algebras is the tensor product $\gradedF(U_1|V_1) \otimes \gradedF(U_2|V_2)$ with canonical inclusions, and the initial object is $\gradedF(\ast|0) \simeq \RR$.
The coproduct of two algebraic models is in general not an algebraic model anymore.

However we already know what answer is reasonable, and we set it as a definition.
If $\calM_i$ are two graded manifolds, then we define $\calM_1 \times \calM_2$ to be the topological space $|\calM_1| \times |\calM_2|$ with structure sheaf defined by
\begin{equation}
\calO_{\calM_1 \times \calM_2} (U_1 \times U_2) =  \gradedF(U_1 \times U_2|V_1 \oplus V_2)
%\gradedF(U_1|V_1) \check{\otimes} \gradedF(U_2|V_2)
\end{equation}
if $U_i$ is an open subset of $|\calM_i|$ such that $\calO_{\calM_i}(U_i) \simeq \gradedF(U_i|V_i)$ as algebra-ed spaces (recall that it is sufficient to define a sheaf on a basis of the topology) and obvious restriction maps.
The product is well-defined because of Lemma~\ref{lem:local}. For details about topology, tensor products, completions, etc. see the appendix
\ref{sec:funcan}.

With a bit more effort one can show that the product defined is a categorical product (see \cite{NP3}):

\begin{prop}
The category of graded manifolds is cartesian monoidal, with terminal object $(\{0\},\RR)$, and the ``body'' functor preserves finite products.
\end{prop}

%\newpage

\section{A note on functional spaces for graded manifolds }
\label{sec:funcan}

\subsection{Motivation}
  Some important work has been done in functional analysis to establish the (weakest possible)
  properties  of functional spaces that still permit to do ``reasonable'' analysis.
  Roughly speaking, the subject is how general one can be in relaxing the hypothesis on the
  considered space of functions and its supporting object, still being able to make sense
  of the usual operations coming from differentiable functions on, say, $\R^n$. This resulted in a series
  of publications/books (starting probably from the fifties), with keywords like \emph{Hilbert}, \emph{Banach}, \emph{Fr\'echet},  \emph{nuclear} spaces...

  The purpose of this appendix is to study the situation for graded manifolds and fit it to the well-established functional analytic framework,
  in order to be able to work in a local charts not bothering about various convergence issues.
  More precisely, we are considering the \emph{local model} for a sheaf of functions on a graded manifold:
  ${\cal C}({\cal U}) = C^{\infty}(U) \compTens S(V)$, where $U \subset \R^n$ is an open set and $V$ is a $\Z$-graded vector space,
  $S(\cdot)$ denotes the sheaf of (graded!) commutative algebras freely generated by $V$.
  We write $\compTens$
  to stress the fact that we consider formal power series (not just polynomials)
  with coefficients in smooth functions on $U$;
  the monomials in these series
 depend on variables defined by $V$, satisfying appropriate commutation relations given by the grading -- all this will be detailed in the sequel.
    We will discuss the topology on this space and show that it behaves nicely with respect to
  usual operations. %, like tensor products.

  The intuition behind is related to several known concepts from classical (non-graded) functional analysis:
  \begin{itemize}[itemsep=-0.1em]
   \item Topology of $C^\infty(M)$ -- smooth functions on a (compact smooth) manifold, or $C^\infty(U)$ -- smooth functions on
   an open subset of $\R^n$
   \item   Fourier analysis, where one constructs the basis on a (segment in a) real line $\R^1$
  but recovers the whole functional space on $R^n$ by completing the tensor product.
  \item Topology of $\R^\infty$ (sometimes denoted by $\R^\omega$ or $\R^{\mathbb N}$) -- the space of real-valued sequences.
  \end{itemize}

 We will also say some words about ``globalizing'' the
  result, i.e. promoting the properties of the sheaf of functions from a local chart to the whole graded manifold.

%  \subsection{Properties  of functions on $\Z$-graded manifolds.}
  \subsection{Local model for $\Z$-graded manifolds}

  Consider a graded manifold ${\cal M} = (M, {\cal O}(M))$,\footnote{Curly letters will usually be related to graded objects, while
  straight letters denote either smooth (non-graded) objects or ingredients of the graded ones.} let us describe locally the sheaf of functions.
  Fix an open chart of $M$: $U \in \R^n$ and decompose the graded vector space $V$ in the following way:
  \begin{equation} \label{sumV}
   V = V^{d_{-l}}_{-l} \oplus V^{d_{-l+1}}_{-l+1}\oplus \dots \oplus V^{d_{-1}}_{-1} \oplus \{0\} \oplus V^{d_{1}}_{1} \oplus \dots \oplus V^{d_{k}}_{k}
  \end{equation}
  We assume the graded manifold to be of \emph{finite degree},  i.e. the maximal/minimal degree of generating elements is bounded and
  this decomposition indeed stops in both directions after a finite number of terms. The subscripts $i$ or $j$ of $V_{\bullet}^{\bullet}$ denote the degree of elements of the respective subspace, and
  the superscript $d_i$ or $d_j$ the dimension of $V_{\bullet}^{\bullet}$. We have two families of indices to distinguish between odd ($i$) and even ($j$)
  degrees, since only the parity of the element (not the degree) plays a role in commutation relations and will make an important difference while describing the elements of ${\cal F}({\cal U})$. Denote  for convenience
  $$ D_1 = d_{-1} + d_{-3} + \dots + d_1 + d_3 + \dots = \sum d_i, \quad
  D_2 = d_{-2} + d_{-4} + \dots + d_2 + d_4 + \dots = \sum d_j $$
   respectively ``odd'' and ``even'' rank of ${\cal M}$. The conceptual difference is that the odd variables ($\xi$'s) are self-anticommuting, and thus square to zero,
   while the even ones ($\eta$'s) are self-commuting and can be raised to arbitrary power.  In this way a function $f \in {\cal F}({\cal U})$ expands
   as a formal power series
   \begin{equation} \label{series}
    f = \sum\limits_{\begin{array}{c}
                     i_1, \dots, i_{D_1} \in \Z_2 \\
                     j_1, \dots, j_{D_2} \in \Z_{\ge 0}
                    \end{array}} f_{i_1\dots i_{D_1}j_1\dots j_{D_2}}(\bm x) \xi_1^{i_1}\dots \xi_{D_1}^{i_{D_1}} \eta_1^{j_1}\dots \eta_{D_2}^{j_{D_2}},
   \end{equation}
   where each coefficient $f_{....}(\bm x)$ is a smooth function of $\bm x\in U \subset \R^n$.
   And the whole functional space ${\cal C}({\cal U})$ morally is
    `` $  \left( C^\infty(U) \right)^{2^{D_1} \cdot |\Z_{\ge 0}|^{D_2}}$ '', that is an infinite (but obviously countable!\footnote{One shows that it is countable by the usual Cantor's diagonal procedure, like  countability of $\mathbb{Q}$.})
    line of smooth functions that are ordered lexicographically by $i_1\dots i_{D_1}j_1\dots j_{D_2}$.
    It is important to note that fixing the expansion (\ref{sumV}), % and is trivial at degree $0$,
     guarantees the uniqueness of (\ref{series}) for any $f \in {\cal C}({\cal U})$.

  \subsection{Topology of $C^\infty(U)$, $\R^{\infty}$ and ${\cal C}$}

  \textbf{1. Fr\'echet.}
Let us recall the usual construction of topology on the space (sheaf) of smooth functions on an open set $C^\infty(U)$ (or on a smooth manifold $M$).
$C^{\infty}(U)$ is an $\R$-linear locally convex topological vector space\footnote{\textbf{Def.} A \emph{topological vector space} %$E$
is a vector space s.t. the linear operations are continuous w.r.t. the chosen topology. It is \emph{locally convex} if any non-empty open set contains a convex open subset.},
with the topology that we are going to define.
Because of the linearity it is sufficient to check all the properties around zero.

For any $f \in C^{\infty}(U)$ define
$
p_{\alpha, K} = \sup\limits_{x\in K} \left| \frac{\displaystyle \partial^{|\alpha|}}{\displaystyle\partial^{\alpha} x} f(x) \right|$,
where $K$ is a compact set, \newline
$\alpha = (\alpha_1, \dots, \alpha_n)$ -- a multi-index to encode partial derivatives.
If $K$ is running over a countable  set of compacts covering $U$, the family $\{p_{\alpha, K}\}$ is a countable (say, indexed by $N \in \mathbb{N}$) family of seminorms\footnote{\textbf{Def.} A seminorm on a vector space
is a  real-valued non-negative functional, s.t. $p(g+h) \le p(g) + p(h)$, $p(ag) = |a|p(g)$ (no non-degeneracy assumed).}.
Those seminorms separate points in $C^{\infty}(U)$, i.e. if $g \neq 0$ there is a at least one $p_{\alpha, K}(g) \neq 0$.
Thus, they define a translation-invariant metric
$$
\rho(g, h) := \sum\limits_{N=1}^{\infty} 2^{-N} \frac{p_N(g - h)}{ 1 + p_N(g-h)}.
$$
$\rho(g, h)$, in turn, defines the topology on $C^{\infty}(U)$, that is $C^{\infty}(U)$ is a \textbf{Fr\'echet space}\footnote{\textbf{Def.}  A Fr\'echet space is a locally convex topological vector space, whose topology is induced by some complete translation invariant metric}. (See for example \cite{KF} for details.)

Moreover, it is a \textbf{Fr\'echet algebra}.\footnote{\textbf{Def.} A Fr\'echet algebra is a Fr\'echet space, s.t. it's topology can be defined by a countable family of
(sub)multiplicative seminorms: $p_N(gh) \le p_N(g)p_N(h)$.}
To show that, we consider a family of seminorms $\tilde p_{i,K} = 2^i \sup \limits_{|\alpha| \le i} p_{\alpha, K}$, which due to rescaling by $2^i$
and the product rule for the derivative become submultiplicative.

We can perform a similar (even simpler) construction for $\R^{\infty}$ -- the space of all real-valued sequences.
For a sequence $s \in \R^{\infty}$, the semi-norm $q_M(s) = \max\limits_{i \le M} |s_i|$,
the topology defined in this way corresponds to element-wise
convergence. One can equivalently take a the sum of absolute values, or for finite families just the absolute value of the $M$-th term.
This is actually an example of a class of spaces called \emph{FK} (\emph{Fr\'echet coordinate}) \emph{spaces}.
This is also a  Fr\'echet algebra -- the simplest one from those described in \cite{dales}: it is automatically closed with respect to
multiplication (formal multiplication of power series), so one needs only to check the submultiplicative property of seminorms.
That is trivially satisfied: $q_M$ does not see the powers greater than $M$, and multiplication increases the power.
In \cite{helemsky} such objects are called \emph{polynormed algebras}.

{Just as a side remark, both of these spaces are not Banach\footnote{\textbf{Def.} A Banach space is a complete normed vector space.}:
the given metrics are not defined from norms.
% For example,  a sequence $s_M$ which has an entry of $1$ at the $M$-th slot and zeros elsewhere,
% does not converge to a null sequence, while $\rho(s_M, null) \to 0$.
But both  are limits of Banach spaces, hence are Fr\'echet.}

We can now consider the space of all functional sequences ${\cal F} = $``$(C^{\infty}(U))^{\infty}$''
(or equivalently formal power series with coefficients in smooth functions)
with the seminorms \newline
$p_{N,M} := \max\limits_{m \le M} p_{N,m}$, where $p_{N,m}$ is $p_N$ as above, applied to the functions in the $m$-th slot of the sequence.
This is again a countable family of seminorms, hence, with the metric $\rho(g, h) := \sum\limits_{N=1}^{\infty}\sum\limits_{M=1}^{\infty} 2^{-N}2^{-M} \frac{p_{N,M}(g - h)}{ 1 + p_{N,M}(g-h)}$, we
prove that ${\cal F}$ is a Fr\'echet space.

With the same reasoning it is a Fr\'echet algebra: the semi-norms are submultiplicative in each term
like for $C^\infty(U)$, and when one has non-zero terms in different slots
they behave like above for $\R^\infty$.

Remark: To be on a safe side from the point of view of functional analysis,
for this whole section we need to assume the \emph{Axiom of countable choice}, to be able to
apply the triangular enumeration for countable number of countable sets.

\textbf{2. Nuclear.}
Let us now consider smooth functions on a product of two open sets $C^{\infty}(U_1\times U_2)$,
clearly this is not the same as $C^{\infty}(U_1)\otimes C^{\infty}(U_2)$ (a function of two variables is not necessarily a product of two functions of one variable).
But the completed tensor product $C^{\infty}(U_1)\compTens C^{\infty}(U_2)$ is actually isomorphic to $C^{\infty}(U_1\times U_2)$.
This property (called \emph{fundamental isomorphism}) can be used as a definition of \emph{nuclear spaces} (\cite{Grothendieck}), and
$C^{\infty}(U)$ is nuclear (as well as $C^{\infty}(M)$).

%\begin{center}
 %\rule{7cm}{0.4pt}
%\end{center}
For the sake of `completeness' let us recall here these topological definitions.
The subtlety is related to the possibility of defining a-priori different topologies on the tensor products
(\cite{Schaefer}). Consider  a vector space $E$ and a family of (locally convex topological) vector spaces $\{E_a, \tau_a\}_{a\in A}$ with linear maps $f_a \colon E \to E_a$
and $g_a \colon E_a \to E$.

\emph{Projective topology $\tau_\pi$} on $E$ is the weakest (coarsest), s.t. all
$f_a$ are continuous. For the base of $\tau_\pi$ around $x\in E$ one takes $\bigcap\limits_{a \in H} f^{-1}_{a}(U_a)$, where $U_a$ are the neighborhoods of
the images $x_a=f_a(x)$, $H$ -- finite subset of $A$.
If $A$ is equipped with a (reflexive, transitive, antisymmetric) relation ``$\le$'' -- partial order -- this permits to
define \emph{projective limits}.\newline
Let $g_{ab} \colon E_b \to E_a$ be continuous linear mappings; $E$ -- subspace of $\prod\limits_a E_a$,
consisting of $x$, s.t. $x_a := f_a(x)$ satisfy $x_a = g_{ab} x_b$ for $a \le b$. $E$ is a \emph{projective limit} of $E_a$, denoted by
$\lim\limits_{\leftarrow}g_{ab}E_b$.

       \emph{Inductive topology $\tau_\iota$}
 is the strongest (finest) one, s.t. all $g_a$ are continuous. In a similar way, the base of this topology is given by all
 (radial, convex, rounded\footnote{Let us not go into details defining those.}) subsets $U\subset E$, s.t. $g_a^{-1}(U)$ are
 neighborhoods of zero in $E_a$. Let, like above, ``$\le$'' be a partial order of indeces, and $h_{ba} \colon E_a \to E_b$  -- continuous linear mappings.
Denote  $F := \bigoplus\limits_{a}E_a$ with $g_a$ -- canonical embeddings of $E_a$ into F,
and $H$ -- a subspace spanned by the images of $E_a$ by $g_a - g_b \circ h_{ba}$, $a \le b$.
If $H/F$ is Hausdorff then it is an \emph{inductive limit} of $\{E_a\}_{a\in A}$ with respect to the mappings $h_{ab}$, denoted by $\lim\limits_{\rightarrow}h_{ab}E_b$. The inductive limit is called \emph{strict} if $\tau_a$ induces $\tau_b$ for $b\le a$.

Facts (from \cite{Schaefer}): \begin{itemize}[itemsep = -0.1em]
        \item A projective limit of a family of locally convex complete vector spaces is a locally convex complete space.
        \item Any complete locally convex vector space $E$ is isomorphic to a projective limit of a family of Banach spaces. One can choose this family
        to be of the same cardinality as a given base of neighborhoods of zero in $E$.
        \item (Corollary) Any Fr\'echet space is isomorphic to a projective limit of Banach spaces; any locally convex space is isomorphic to a subspace of a product of Banach spaces.

    \item  A locally convex direct sum of a family of locally convex spaces is complete iff each of them is complete.

    \item A strict inductive limit of a sequence of complete locally convex spaces is a complete locally convex space.
\end{itemize}

A generic topology that one would define  is somewhere between the projective and the inductive ones.
But in good cases (e.g. for \textbf{nuclear spaces}) there is no ambiguity, since completions with respect to both topologies produce
isomorphic results. There are several ways to define nuclear spaces, establishing isomorphisms between
topologies (like in \cite{Schaefer}); or alternatively (equivalently), one can just ask for the fundamental isomorphism to hold (\cite{Grothendieck}).
Other ways include \cite{gelfand-shilov} -- working with variation bounded functionals, \cite{pietsch} -- with less attention to topological tensor products though, and the list is certainly not exhaustive.

%\begin{center}
 %\rule{7cm}{0.4pt}
%\end{center}

Regardless of the choice (of equivalent) definitions one uses the following facts (\cite{Schaefer})
about nuclear spaces hold true: \\[-2.5em]
\begin{enumerate}[itemsep = -0.1em] %, leftmargin=-0.1em]
 \item Any complete nuclear space is isomorphic to a projective limit of some family of Hilbert\footnote{\textbf{Def.} A Hilbert space is a Banach space
 the norm on which is defined by some positive definite scalar product.} spaces.
 A Fr\'echet space is nuclear iff it can be represented as a projective limit of Hilbert spaces $E  = \lim\limits_{\leftarrow} g_{mn} H_n$,
 s.t. $g_{mn}$ are nuclear maps\footnote{\label{nuclmap}The axiomatic definition of a nuclear map between two linear spaces $E$ and $F$
 is a bit technical (see again \cite{Schaefer}), but it amounts to the following description: A linear map $u \colon E\to F$ is nuclear iff it is of the form  %\newline
 $u(x) = \sum\limits_{n=1}^\infty \lambda_b f_n(x) y_n = \sum\limits_{n=1}^\infty \lambda_b f_n \otimes y_n$, where  $\sum\lambda_n$ is an absolutely converging series, $f_n$ is an equicontinuous sequence in
 $E^*$, $y_n$ is a sequence contained in a convex rounded and bounded subset $B \subset U$, s.t. $F_B$ is complete.
 ($F_B := \bigcup\limits_{n = 1}^\infty nB$, with the Minkowski functional as a norm)}  for $m<n$.
 \item (Theorem) Any subspace and any separated quotient space of a nuclear space is nuclear.  A product of any family of nuclear spaces is nuclear,
 a locally convex direct sum of a countable family of nuclear spaces is nuclear.
 \item (Corollary) Projective limit of any family of nuclear spaces is nuclear.
 \item (Corollary) Inductive limit of a countable family of nuclear spaces is nuclear.
 \end{enumerate}

These properties (especially 2.) are already more than sufficient to say that the space of functional sequences -- ${\cal F}$ from above is nuclear, since it is
a limit of a countable family $\{ (C^{\infty}(U))^N \}_{N\in\mathbb{N}}$ with obvious embeddings of $(C^{\infty}(U))^N \subset (C^{\infty}(U))^{N+1}$,
or a product of a family of a countable number of copies of $C^{\infty}(U)$.

Alternatively, one can do it ``by hand'': $C^{\infty}(U_1)\hat\otimes C^{\infty}(U_2) \cong C^{\infty}(U_1\times U_2)$ in each term of the sequence,
and the terms do not interact, i.e. this is true for the whole space of sequences.
This is roughly speaking the idea of the proof of a part of item 2: one uses the form of the nuclear map
given in the footnote \ref{nuclmap}, then introduces a second index responsible for the number of the term of a sequence and checks
that the desired properties of this sequence are satisfied.

 %\newpage

\subsection{Application to $\Z$-graded manifolds.}
As described above the local model of the space of functions on a graded manifold after fixing the structure of the graded vector space reduces
to a sequence of smooth functions on an open set, which in view of the previous section is a Fr\'echet nuclear space.
It is even a Fr\'echet algebra for the same argument as in \cite{dales}: as soon as the (lexicographical)
order is fixed for the monomials in the series, the multiplication follows the same logic as for ordinary
power series.
% \VS{Is there any subtlety with negative gradings in  (\ref{sumV})?}

And all this is again visible ``by hand''.
For instance concerning nuclearity, consider the product of two graded manifolds: ${\cal M}\times\tilde{\cal M}$, with
$(M, V)$, $(\tilde M, \tilde V)$ as before. For the explicit expression of $f\in {\cal F}({\cal U} \times \tilde{\cal U})$
one fixes again some order of powers of elements in  $V\times \tilde V$, that produces strings like
$i_1\dots i_{D_1+\tilde D_1}j_1\dots j_{D_2+\tilde D_2}$.
Since there is no need to make it canonically, one can fix an appropriate basis of $V\times \tilde V$, induced by the bases of
$V$ and $\tilde V$. Hence those strings can be naturally decoupled to
$i_1\dots i_{D_1}\tilde i_1\dots \tilde i_{\tilde D_1}j_1\dots j_{D_2}\tilde j_1\dots \tilde j_{\tilde D_2}$, giving
$i_1\dots i_{D_1}j_1\dots j_{D_2}$
and $\tilde i_1\dots \tilde i_{\tilde D_1}\tilde j_1\dots \tilde j_{\tilde D_2}$. This reduces the fundamental isomorphism problem  to
(countably many) independent $C^{\infty}(U)\compTens C^{\infty}(\tilde U) \cong C^{\infty}(U\times \tilde U)$.

\begin{rmk}
Degression to $\Z_2^n$-graded manifolds:
a careful treatment of the above issues in the case of $\Z_2^n$-graded manifolds can be found in \cite{NP2}.
\end{rmk}

%The same logic applies to the case of $\Z_2^m$-graded manifolds: one needs to remove negatively graded subspaces in the expansion (\ref{sumV}) of $V$,  and
%replace positively graded by multigraded. After computing the parity appearing in the commutation relations one obtains literally the same expression
%as (\ref{series}), modulo an obvious remark on Fr\'echet algebras.

%\subsection{Globalization}
%To go from local picture $({\cal U}, {\cal O}_{\cal U})$ to the whole graded manifold $({\cal M}, {\cal O}_{\cal M})$, there is
%apparently no problem. But to make things clean
%one may want to go to some sheaf-theoretic arguments. Relevant references here are  \cite{schapira_lec} and \cite{cat_sh}
%chapter 19: about stacks and statements like local representability for good assumptions implies global representability.

\newpage
%%%%%%%%%%%%%%%%%%%%%%%%%%%%%%%%%%%%%%%%%%%%%%%%%%%%%%%%%%%%%
%%%%%%%%%%%%%%%%%%%%%%%%%%%%%%%%%%%%%%%%%%%%%%%%%%%%%%%%%%%%%
%%%%%%%%%%%%%%%%%%%%%%%%%%%%%%%%%%%%%%%%%%%%%%%%%%%%%%%%%%%%%

\end{document}